\PassOptionsToPackage{unicode}{hyperref}
\PassOptionsToPackage{hyphens}{url}
\PassOptionsToPackage{dvipsnames,svgnames,x11names}{xcolor}
\documentclass[12pt]{article}
\usepackage{dsfont}
\usepackage{iftex}
\ifPDFTeX
  \usepackage[T1]{fontenc}
  \usepackage[utf8]{inputenc}
  \usepackage{textcomp} % provide euro and other symbols
\else % if luatex or xetex
  \usepackage{unicode-math}
  \defaultfontfeatures{Scale=MatchLowercase}
  \defaultfontfeatures[\rmfamily]{Ligatures=TeX,Scale=1}
\fi
\usepackage{lmodern}
\ifPDFTeX\else  
\fi
\IfFileExists{upquote.sty}{\usepackage{upquote}}{}
\IfFileExists{microtype.sty}{% use microtype if available
  \usepackage[]{microtype}
  \UseMicrotypeSet[protrusion]{basicmath} % disable protrusion for tt fonts
}{}
\makeatletter
\@ifundefined{KOMAClassName}{% if non-KOMA class
  \IfFileExists{parskip.sty}{%
    \usepackage{parskip}
  }{% else
    \setlength{\parindent}{0pt}
    \setlength{\parskip}{6pt plus 2pt minus 1pt}}
}{% if KOMA class
  \KOMAoptions{parskip=half}}
\makeatother
\usepackage{xcolor}
\usepackage{comment}
\makeatletter
\ifx\paragraph\undefined\else
  \let\oldparagraph\paragraph
  \renewcommand{\paragraph}{
    \@ifstar
      \xxxParagraphStar
      \xxxParagraphNoStar
  }
  \newcommand{\xxxParagraphStar}[1]{\oldparagraph*{#1}\mbox{}}
  \newcommand{\xxxParagraphNoStar}[1]{\oldparagraph{#1}\mbox{}}
\fi
\ifx\subparagraph\undefined\else
  \let\oldsubparagraph\subparagraph
  \renewcommand{\subparagraph}{
    \@ifstar
      \xxxSubParagraphStar
      \xxxSubParagraphNoStar
  }
  \newcommand{\xxxSubParagraphStar}[1]{\oldsubparagraph*{#1}\mbox{}}
  \newcommand{\xxxSubParagraphNoStar}[1]{\oldsubparagraph{#1}\mbox{}}
\fi
\makeatother

\usepackage{longtable,booktabs,array}
\usepackage{calc} % for calculating minipage widths
\usepackage{etoolbox}
\makeatletter
\patchcmd\longtable{\par}{\if@noskipsec\mbox{}\fi\par}{}{}
\makeatother
\IfFileExists{footnotehyper.sty}{\usepackage{footnotehyper}}{\usepackage{footnote}}
\makesavenoteenv{longtable}
\usepackage{graphicx}
\makeatletter
\def\maxwidth{\ifdim\Gin@nat@width>\linewidth\linewidth\else\Gin@nat@width\fi}
\def\maxheight{\ifdim\Gin@nat@height>\textheight\textheight\else\Gin@nat@height\fi}
\makeatother
\setkeys{Gin}{width=\maxwidth,height=\maxheight,keepaspectratio}
\makeatletter
\def\fps@figure{htbp}
\makeatother

\makeatletter
\@ifpackageloaded{caption}{}{\usepackage{caption}}
\AtBeginDocument{%
\ifdefined\contentsname
  \renewcommand*\contentsname{Table of contents}
\else
  \newcommand\contentsname{Table of contents}
\fi
\ifdefined\listfigurename
  \renewcommand*\listfigurename{List of Figures}
\else
  \newcommand\listfigurename{List of Figures}
\fi
\ifdefined\listtablename
  \renewcommand*\listtablename{List of Tables}
\else
  \newcommand\listtablename{List of Tables}
\fi
\ifdefined\figurename
  \renewcommand*\figurename{Figure}
\else
  \newcommand\figurename{Figure}
\fi
\ifdefined\tablename
  \renewcommand*\tablename{Table}
\else
  \newcommand\tablename{Table}
\fi
}
\@ifpackageloaded{float}{}{\usepackage{float}}
\floatstyle{ruled}
\@ifundefined{c@chapter}{\newfloat{codelisting}{h}{lop}}{\newfloat{codelisting}{h}{lop}[chapter]}
\floatname{codelisting}{Listing}

\makeatother
\makeatletter
\@ifpackageloaded{caption}{}{\usepackage{caption}}
\@ifpackageloaded{subcaption}{}{\usepackage{subcaption}}
\makeatother

\ifLuaTeX
  \usepackage{selnolig}  % disable illegal ligatures
\fi
\usepackage[]{natbib}
\usepackage{bookmark}

\IfFileExists{xurl.sty}{\usepackage{xurl}}{} % add URL line breaks if available
\hypersetup{
  pdftitle={Title},
  pdfauthor={Author 1; Author 2},
  pdfkeywords={3 to 6 keywords, that do not appear in the title},
  colorlinks=true,
  linkcolor={blue},
  filecolor={Maroon},
  citecolor={Blue},
  urlcolor={Blue},
  pdfcreator={LaTeX via pandoc}}

\newcommand{\anon}{1}
\usepackage{multirow}

\usepackage{amsthm}  % Needed for theorem-like environments

\usepackage{algorithm, algorithmic}

\newtheorem{theorem}{\bf Theorem}[section]
\newtheorem{lemma}{\bf Lemma}[section]

\newtheorem{remark}{\it Remark}[section]
\newtheorem{proposition}{Proposition}
\newtheorem{example}{Example}
\usepackage{float}
\usepackage{tabularx, booktabs, array}

\RequirePackage{amsmath,amsfonts,amssymb, bm,mathtools}
\begin{document}

\def\spacingset#1{\renewcommand{\baselinestretch}%
{#1}\small\normalsize} \spacingset{1}

%%%%%%%%%%%%%%%%%%%%%%%%%%%%%%%%%%%%%%%%%%%%%%%%%%%%%%%%%%%%%%%%%%%%%%%%%%%%%%

\if1\anon
{
  \title{\bf %Beyond Global Marginal Coverage: 
Conformal Prediction Intervals with Tail-Specific Guarantees 
% One-sided conformal prediction intervals
% Beyond Marginal Coverage: Intersecting One-Sided Conformal Intervals for Tail-Specific Guarantees
%Enhancing Conformal Prediction for a More Reliable Global Coverage with Tail-Specific Guarantees
%Intersecting One-sided Conformal Intervals for a More Reliable Global Coverage with Tail-Specific Guarantees
}
  \author{Simone Cuonzo\thanks{
    The authors gratefully acknowledge support by Sapienza Research Grants, Ateneo Piccoli
2024–RP1241905F06F855 and Avvio alla Ricerca 2025-AR125199C38C3858. In
accordance with COPE guidelines and the TITAN Guideline Checklist 2025, we acknowledge the use of ChatGPT and Claude for MAC v1.9659.2 for linguistic editing. No content generation or interpretation was delegated to the tool. The authors remain fully responsible for the manuscript. No artificial intelligence tools have been used for image generation, data
collection, or data analysis.}\hspace{.2cm}\\
    MEMOTEF Department, Sapienza University of Rome\\
    and \\
    Nina Deliu \\
    MEMOTEF Department, Sapienza University of Rome\\
    MRC--Biostatistics Unit, University of Cambridge}
  \maketitle
} \fi

\if0\anon
{
  \bigskip
  \bigskip
  \bigskip
  \begin{center}
    {\LARGE\bf Title}
\end{center}
  \medskip
} \fi

\bigskip
\begin{abstract}
% The text of your abstract. 200 or fewer words.

This paper extends classical conformal frameworks for constructing prediction intervals with {\it global} marginal coverage $1-\alpha$ to intervals that provide explicitly calibrated guarantees for the upper and lower tails separately. 
 Focusing on split conformal prediction, 
we first construct lower and upper one-sided conformal intervals that achieve marginal validity, and then  
derive the induced two-sided interval by intersection. Theoretical results prove both tail-specific and global marginal coverage of the induced two-sided interval. Results are presented first for the exchangeable setting, where coverage has finite-sample guarantees, and then for non-exchangeable data, where guarantees are asymptotic. Simulation studies show that the proposed approach achieves improved directional calibration relative to classical two-sided intervals, especially relevant in skewed data. Finally, the benefit of the proposed framework is showcased in a financial application, where one aims for return maximization while seeking strict control on the left tail.
\end{abstract}

\noindent%
{\it Keywords:} Conformal prediction for skew data, Finite-sample tail coverage, Reliable risk management, Robust conformal prediction, Uncertainty quantification
\vfill

\newpage
\spacingset{1.8} % DON'T change the spacing!

%\vspace*{-.2cm}
\section{Introduction}
\vspace*{-.1cm}

Consider a set of $n$ {\it training} data $\mathcal{D}_n = {(X_i, Y_i)}_{i=1}^{n}$, where each $i$-th pair consists of a feature vector $X_i \in \mathbb{R}^p$, $p \ge 1$, and an associated scalar response $Y_i \in \mathcal{Y} \subseteq \mathbb{R}$. Let $X_{n+1}$ denote a (new) {\it test} feature. Having observed $\mathcal{D}_n$ and $X_{n+1}$, the objective is to construct a prediction interval $\mathcal{C}_{n+1} \coloneq \mathcal{C}(X_{n+1}) \subseteq \mathcal{Y}$ that is likely to contain, or ``{\it cover}'', the true value of the unseen test response $Y_{n+1}$. Conformal prediction (CP) offers a {\it valid} yet flexible distribution-free framework for this predictive inference problem, under the sole assumption of {\it exchangeability}~\citep{vovk2005algorithmic, lei2018distribution}. Specifically, if $(X_1,Y_1),\ldots,(X_{n+1},Y_{n+1})$ are exchangeable, a CP interval is guaranteed to satisfy
\begin{equation}
\mathbb{P}\left( Y_{n+1} \in \mathcal{C}_{n+1} \right) \ge 1 - \alpha,\qquad \alpha \in (0,1),
\label{eq:validity}
\end{equation}
where the probability is taken over the joint distribution of the training data and the test observation, in line with the frequentist philosophy. For a comparison with prediction intervals under the Bayesian paradigm, see~\cite{deliu2026interplay}. This frequentist guarantee is often referred to as {\it marginal coverage} or {\it statistical validity}. Importantly, this property holds in {\it finite samples}, without assumptions on the form of the underlying distribution. Moreover, it remains compatible with arbitrarily complex predictive models, including modern regression models as well as machine learning (ML) algorithms. 

The versatility of CP has contributed to its growing popularity as a tool to quantify uncertainty in modern statistical practice. Yet, standard CP methods are not without limitations. For example, the marginal guarantee in Eq.~\eqref{eq:validity} does not extend to conditional coverage given $X_{n+1}$ and exact conditional coverage is impossible without further (distributional) assumptions \citep{ lei2014distribution, barber2021limits}. The most practical remedies rely on adaptive nonconformity scores, such as {\it conformalized quantile regression}~\citep[CQR;][]{romano2019conformalized}, which produce intervals that are more responsive to local heteroscedasticity. Another challenge concerns the exchangeability assumption, which fails in many realistic settings, ranging from covariate shift \citep{tibshirani2019conformal, qin2025distribution} to temporal dependence~\citep{gibbs2021adaptive}; more general forms of distribution drift and corresponding remedies are discussed in \citet{barber2023conformal}.

The challenges outlined above all concern a {\it global} or average form of coverage of the entire response range, with no regard for direction. In this work, we identify and address a complementary challenge: in many real-world prediction and decision problems, the two tails of the predictive distribution are inherently asymmetric and carry different consequences. Two paradigmatic examples illustrate this point.
%Therefore, both tails deserve their own guarantee, not just overall coverage. 
\begin{example} In portfolio management, an investor seeks to maximize returns while requiring explicit protection against extreme losses. Both positive and negative deviations from the expected return must be characterized, but the downside carries far more severe consequences, demanding stricter probabilistic guarantees. Value-at-Risk formalizes this asymmetry by focusing on the left tail of the distribution~\citep{mcneil2015quantitative}. Two-sided risk measures extend this idea to jointly consider the two tails~\citep{chen2008two, farinelli2008beyond}.
\end{example}

\begin{example} A similar asymmetry arises in medicine: when monitoring a patient's blood pressure, a clinician must guard against both extremes, which reflect hypotension and hypertension; yet, the emphasis often has a direction. In post-surgical hypertensive patients, for instance, the primary concern is to detect dangerous spikes above a safety threshold, as values that are too high may trigger stroke. %In both settings, the decision-maker cares about both tails, but one tail demands a stricter probabilistic guarantee. Often, string probabilistic control is needed over both tails.
Analogously, in early-phase clinical trials, a new treatment must satisfy simultaneous constraints on toxicity and efficacy, where the two thresholds carry inherently different stakes~\citep{thall2004dose}. In all such settings, the decision-maker must account for both tails, but each demands its own probabilistic guarantee.
%\textcolor{red}{valuare anche questo articolo per toxicity-efficacy in phase 1-2 trials: https://jmlr.org/papers/v22/19-228.html}
\end{example}

Common methods for constructing prediction intervals, including the CP framework, do not directly respond to this two-directional objective. They distribute coverage across both tails, seeking for the {\it marginal guarantee} in Eq.~\eqref{eq:validity}, which ensures that the prediction interval contains the true response with high probability on average, but provides no explicit control over individual tail probabilities. In particular, it does not distinguish between coverage failures in the lower tail and those in the upper tail, leaving directional calibration unaddressed. This motivates the development of more robust prediction intervals that deliver separate and explicitly calibrated guarantees for each tail, allowing practitioners to assign stricter coverage requirements where the cost of error is more severe.

\subsection{Specific Contributions and Organization}

This work proposes a CP framework that strengthens classical marginal coverage with explicitly calibrated guarantees for the upper and lower tails of $Y_{n+1}$ separately. Focusing on split CP~\citep{vovk2005algorithmic}, this is achieved by constructing one-sided CP intervals for each tail and recovering a two-sided interval through their intersection. The specific contributions of this work are fourfold.
\begin{description}
\item[C1] We introduce a stronger notion of marginal validity with \textit{tail-specific coverage} (Section~\ref{sec: preliminaries}) and propose a general CP procedure that allows explicit control on the left and right tails with user defined miscoverage rates $\alpha^-$ and $\alpha^+$, respectively (Section~\ref{sec:onesided_split_cp}). We present a variation of the quantile score that retains the sign of the residual relative to the estimated quantile, in contrast to the truncated version used in standard CQR.

\item[C2] In Section~\ref{sec: theory}, we establish theoretical guarantees for the proposed CP framework, proving finite-sample tail-specific and marginal coverage under exchangeability (Section~\ref{sec: theory_exch}), and asymptotic counterparts under non-exchangeable data (Section~\ref{sec: theory_nonexch}).
\item[C3] We illustrate the empirical benefits of the proposed framework through simulation studies (Section~\ref{sec: sim_studies}), demonstrating improved directional calibration relative to classical two-sided CP intervals, with particular gains in skewed and heavy-tailed settings, and show the advantage of the signed quantile score over its truncated counterpart.  
\item[C4] In Section~\ref{sec: app}, we showcase the framework in a financial application ({\bf Example 1}), where the asymmetry between downside risk and upside return motivates strict probabilistic control over the left tail alongside a robust indication of potential gains.
\end{description}

\section{Preliminaries} \label{sec: preliminaries}

\subsection{Problem Statement}

Given a training set $\mathcal{D}_n = {(X_i, Y_i)}_{i=1}^{n}$ %with each $X_i \times Y_i \in \mathbb{R}^p \times \mathbb{R}$, 
and a test feature $X_{n+1}$, our goal is to construct a prediction interval for the new response $Y_{n+1}$, say $\mathcal{C}_{n+1} \subseteq \mathcal{Y}$, that not only satisfies the marginal coverage guarantee of Eq.~\eqref{eq:validity}, but also provides explicitly calibrated guarantees on each tail of the distribution of $Y_{n+1}$. Assuming $\mathcal{C}_{n+1}$ takes the form of a single interval $\mathcal{C}_{n+1} = [L_{n+1}, U_{n+1}]$ -- which holds whenever the response variable has a unimodal distribution, but also under less restrictive conditions that will be discussed in Section~\ref{sec:onesided_split_cp} -- the joint set of guarantees we seek is
\begin{align}
\mathbb{P}\left( Y_{n+1} \in [L_{n+1}, U_{n+1}] \right) &\ge 1 - \alpha, \tag{marginal coverage}\\
\mathbb{P}\left( Y_{n+1} > U_{n+1}\right) \le \alpha^+\qquad \text{and}\qquad &\mathbb{P}\left( Y_{n+1} < L_{n+1}\right) \le \alpha^-, \tag{tail-specific coverage}
\end{align}
with $\alpha^-+\alpha^+ = \alpha \in (0,1)$. The tail-specific coverage condition is strictly stronger than marginal coverage: it decomposes the total miscoverage $\alpha$ into directional components, ensuring that no tail is systematically under-protected.

This objective relates to a classical distinction between {\it equal-tailed intervals}, which allocate $\alpha/2$ to each tail, {\it highest density regions}, which yield the shortest interval at a given coverage level~\citep{hyndman1996computing, deliu2024alternative}, and {\it intervals symmetric} around a point predictor. The tail-specific intervals proposed here occupy a different point in this landscape: they target directional reliability, allowing the assignment of asymmetric miscoverages, $\alpha^-$ and $\alpha^+$, to reflect the asymmetric consequences of errors in each tail. This is particularly valuable when the response is skewed or heavy-tailed, settings where a single global coverage conceals meaningful asymmetries between the two tails.

In what follows, we first review the standard split CP framework in both the exchangeable and non-exchangeable setting, before developing the proposed tail-specific construction.

\subsection{Background on Split Conformal Prediction} \label{sec: split-CP}

There are two main approaches to implement conformal prediction: full CP and split CP~\citep{vovk2005algorithmic, fontana2023conformal}. While the former is computationally demanding and consequentially rarely applicable in many practical settings, split CP proposes a computationally efficient procedure that decouples model training and calibration in a similar flavor to the validation-set approach~\citep{james_witten_hastie_tibshirani2021}. 
The most common approach for implementing CP is known as the {\it split CP}~\citep{vovk2005algorithmic}, which can be summarised as follows. The available data are partitioned into two disjoint subsets, $\mathcal{D}_{n_{t}}$, a proper training set of cardinality $n_{t}$ on which a fitted model $\widehat{f}_{n_{t}}$ is obtained, and $\mathcal{D}_{n_{c}}$, a calibration set of size $n_{c}$ which is used to compute the so-called calibration or conformity scores. The central feature of CP, compared to, e.g., asymptotic prediction intervals or resampling methods, is, in fact, the use of a \textit{(non-)conformity score function}, say $s\!: \mathcal{X} \times \mathcal{Y} \to \mathbb{R}$, that quantifies the (dis-)similarity or {\it (non-)conformity} of any point $(X, Y)$ to an observed sample $\mathcal{D}$. In a regression setting, it can be thought of as a residual function: $s_{n_{t}}((x, y)) = |\widehat{f}_{n_{t}}(x) - y|$, with $\widehat{f}_{n_{t}}$ a fitted regression model on $\mathcal{D}_{n_{t}}$.

An important property of a conformity function is that it is symmetric in $\mathcal{D}$, i.e., the conformity for $(x, y)$ is invariant to any permutation of the elements of $\mathcal{D}$. By leveraging the symmetry of $s$ together with the {\it exchangeability} of the $(X_i, Y_i) \in \mathcal{D}_{n_{c}}$, conformity scores computed on $\mathcal{D}_{n_{c}}$ with respect to a model fitted on $\mathcal{D}_{n_{t}}$, i.e.,
\begin{align*}
    S_i = s_{n_{t}}\big((X_i, Y_i)\big),\quad i = 1,\dots, n_{c},
\end{align*}
are themselves exchangeable, ensuring validity. For a non-conformity score function $s$ (e.g., residual function), split CP intervals are obtained as
\begin{equation*}
\mathcal{C}_{n+1,\alpha}
=
\big\{y: s_{n_{t}}\big((X_{n+1}, y)\big) \leq Q_{\alpha, n_{c}}
\big\},
\label{eq:split_cp_general}
\end{equation*}
where $Q_{\alpha, n_{c}}$ denotes the empirical
$(1-\alpha)(1+1/n_{c})$-quantile of the $n_{c}$ calibration scores $S_1,\dots, S_{n_{c}}$. The $(1 + 1/n_{c})$ adjustment enables finite-sample guarantees in Eq.~\eqref{eq:validity}.

While any conformity score ensures validity, their choice plays a crucial role in determining the efficiency and adaptivity of the resulting CP intervals. For example, residual scores yield prediction intervals that are symmetric around the point predictor $\widehat f_{n_{t}}$, implicitly assuming homoscedasticity. Standardized residual (S-residual) scores address this by rescaling residuals according to their conditional standard deviation estimate $\widehat{\sigma}_{n_{t}}(x)$, producing intervals that adapt their width to the local variability of $Y|X=x$. Nevertheless, while adapting to heteroscedasticity, both constructions remain symmetric around $\widehat f_{n_{t}}$ and cannot capture asymmetries in the conditional distribution of $Y|X=x$, potentially leading to systematic under-coverage in one tail. As a solution, \textit{Conformalised Quantile Regression}~\citep[CQR;][]{romano2019conformalized} uses a conformity score constructed on estimated conditional quantiles for both the upper $q_{1-\frac{\alpha}{2}, n_{t}}$ and lower $q_{\frac{\alpha}{2}, n_{t}}$ tails; this is technically characterized in Supplementary material A. Table~\ref{tab:cp_scores}
summarizes these three popular conformity scores, together with the corresponding prediction intervals.

\begin{table}[htbp]
\centering
\caption{Popular conformity scores in CP and associated split CP intervals.}
\label{tab:cp_scores}
\small
\begin{tabularx}{\textwidth}{>{\setlength{\baselineskip}{0.8\baselineskip}}p{0.16\textwidth}
                              >{\setlength{\baselineskip}{0.8\baselineskip}}X
                              >{\setlength{\baselineskip}{0.8\baselineskip}}p{0.55\textwidth}}
\toprule
\textbf{Score name} & \textbf{Score definition 
} & \textbf{Associated Split CP Interval} \\
\hline
Residual $s^{\text{res}}$
&
$  
| \widehat f_{n_{t}}(x) - y |$
&
$\mathcal{C}^{\text{res}}_{\alpha}(x) = \widehat f_{n_{t}}(x) \pm Q^{\text{res}}_{\alpha, n_{c}}$
\\[6pt]

S-residual $s^{\text{s-res}}$
&
$%s^{\text{s-res}}_{n_{t}}\big((x, y)\big) = 
\dfrac{| \widehat f_{n_{t}}(x) - y |}{\widehat{\sigma}_{n_{t}}(x)}$
&
$\mathcal{C}^{\text{s-res}}_{\alpha}(x) = \widehat f_{n_{t}}(x) \pm \widehat{\sigma}_{n_{t}}(x)\, Q^{\text{s-res}}_{\alpha, n_{c}}$
\\[8pt]

Quantile $s^{\text{q}}$
& 
$\max\!\Big\{$\parbox[t]{0.20\textwidth}{$\widehat q_{\frac{\alpha}{2}, n_{t}}(x) - y,$\\[-3pt]
\hspace{0.5em}$y - \widehat q_{1-\frac{\alpha}{2}, n_{t}}(x)\Big\}$}
&
$\mathcal{C}^{\text{q}}_{\alpha}(x) = \big[
\widehat q_{\frac{\alpha}{2}, n_{t}}(x) - Q^{\text{q}}_{\alpha, n_{c}},\;
\widehat q_{1-\frac{\alpha}{2}, n_{t}}(x) + Q^{\text{q}}_{\alpha, n_{c}}
\big]$
\\
\hline
\end{tabularx}
\end{table}

\subsection{Adaptive Conformal Inference in Non-exchangeable Data}
The finite-sample validity guarantees established in the previous section rely
critically on the exchangeability of the data (specifically, the calibration set). However, this
assumption is typically violated in the presence of temporal dependence or distributional shifts.
In such settings, exact marginal coverage can no longer be ensured with standard CP.

To overcome this limitation, \citet{gibbs2021adaptive} introduce the
\emph{Adaptive Conformal Inference} (ACI) framework, which replaces the fixed
miscoverage level $\alpha$ with an adaptive sequence of $\alpha_i$'s that is updated online so as to track the desired long-run violation rate $\alpha$. 
Given an observed set of $n$ data points $\mathcal{D}_n$, for each new pair observed $\{(X_{i},Y_{i})\}_{i\geq n+1}$, one computes the miscoverage or error indicator $\mathrm{err}_{i} := \mathbf{1}\{Y_{i} \notin \mathcal{C}_{i,\alpha_i}\}$, with $\mathcal{C}_{i,\alpha_i}$ a standard CP interval for $X_i$ at the level $1-\alpha_i$, and updates the miscoverage level according to the adaptive recursion
\begin{equation*}
\alpha_{i+1}
=
\alpha_{i}
+
\gamma\bigl(\alpha - \mathrm{err}_{i}\bigr),
\label{eq:aci_lower_general}
\end{equation*}
where $\gamma > 0$ is a step-size parameter controlling the speed of adaptation,
and $\alpha \in (0,1)$ is the target long-run violation rate. 
In this way, the sequence $\{\alpha_i\}_{i \geq n+1}$ is recursively adapted online so as to track the target miscoverage level $\alpha$, using all data pairs $\{(X_j,Y_j)\}_{j=1}^i$.
A key property of the ACI recursion is that it guarantees asymptotic control
of the empirical miscoverage frequency without requiring exchangeability~\citep{gibbs2021adaptive}.

A key limitation of the ACI framework lies in the selection of the learning-rate parameter $\gamma$, which may heavily affect performance in practice: small values of $\gamma$ yield stable but slow adaptation, whereas larger values enable faster tracking at the cost of increased variability. To address this issue, \citet{zaffran2022adaptive} and \citet{gibbs2024conformal} propose the \emph{Online Expert Aggregation on ACI} (AgACI) and the \emph{Dynamically-tuned Adaptive Conformal Inference} (DtACI), respectively, which eliminate the need to specify a single learning rate by aggregating a set of $k$ ACI instances run in parallel. For example, in DtACI each ACI instance is associated with a candidate learning rate $\gamma_j$, $j=1,\dots,k$, and produces its own sequence of adaptive
miscoverage levels $\{\alpha_i^{j}\}_{i\ge n+1}$. 
The aggregation is then performed through an exponential weighting scheme, where, for each $i \geq n+1$, each candidate $j$ is assigned a
weight $w_i^j$ that reflects its past predictive performance. %as measured by the pinball loss. 
These weights are normalized to form a probability distribution
\[
p_i^j := \frac{w_i^j}{\sum_{l=1}^k w_i^l}, \qquad j=1,\dots,k,\qquad i\ge n+1,
\]
which is then used to combine the candidate predictors. %Given these probabilities, 
The algorithm selects $\alpha_i = \alpha_i^j$ with probability $p_i^j$. Equivalently, one can consider the mixture induced by $\{p_i^j\}_{j=1}^k$ as defining an aggregated adaptive miscoverage level given by $\bar{\alpha}_i := \sum_{j=1}^k p_i^j \,\alpha_i^j$. The resulting value $\alpha_i$ (or $\bar{\alpha}_i$) is then used to construct the
conformal prediction set at time $i$.

After observing $(X_i, Y_i)$, the performance of each candidate is evaluated via the
pinball loss $\ell_\alpha(\beta_i,\alpha_i^j)
:=
\alpha(\beta_i-\alpha_i^j)-\min\{0,\beta_i-\alpha_i^j\}$, where $\beta_i
:=
\sup\Bigl\{\beta\in[0,1]:\, Y_i\in \mathcal{C}_{i,\beta}\Bigr\}$ denotes the largest miscoverage level for which $Y_{i+1}$ is contained in the prediction set. The weights
are then updated according to the exponential reweighting rule
\[
\tilde{w}_{i+1}^j
=
w_i^j \exp\! \bigl(-\eta\,\ell(\beta_i,\alpha_i^j)\bigr),\qquad i\ge n+1,
\]
where $\eta>0$ is a learning parameter controlling the sensitivity to losses. 

In parallel, each candidate ACI instance is updated independently according to
\[
\alpha_{i+1}^j
=
\alpha_i^j
+
\gamma_j\bigl(\alpha- \mathrm{err}_i^j\bigr)\qquad i\ge n+1,
\]
where $\mathrm{err}_i^j := \mathbf{1}\{Y_i \notin C_{i,\alpha_i^j}\}$
denotes the miscoverage indicator associated with candidate $j$. We initialize $\alpha^j_{n+1}=\alpha$ and $w_{n+1}^j=1$ for all $j=1,\dots,k$.

This construction preserves the flexibility of the original ACI framework while providing robustness with respect to the choice of the learning-rate parameter. Importantly, in addition to ensuring asymptotic marginal coverage, DtACI also enjoys explicit regret guarantees (not available for ACI or AgACI) with respect to the best fixed learning rate. Therefore, among the adaptive procedures, in this work, we focus on ACI and DtACI, for their coverage and regret guarantees. We also recall that all adaptive procedures remain agnostic to the specific conformity score, allowing for any choice of preference.

\section{Achieving Tail-specific Guarantees in CP} \label{sec:onesided_split_cp}

We now show that the construction of a CP interval with both marginal and tail-specific guarantees can be achieved in two steps: (i) first, one-sided CP intervals of the form $(-\infty,U]$ and $[L,+\infty)$, with guarantees on each tail, respectively, are derived; (ii) a two-sided interval is recovered by their intersection as $[L,U]$. We assume and consider conformity score functions that are quasi-convex in $y$, a sufficient condition for an interval to be a single interval of the forms $(-\infty,U]$, $[L,+\infty)$, or $[L,U]$, rather than a union of disjoint intervals. 

To obtain (i) one-sided intervals, one can simply adjust the conformity scores in Table~\ref{tab:cp_scores}. Consider, for example, the one-sided quantile score with
(lower-tail) miscoverage level $\alpha$
\begin{align} \label{eq: cqr_onesided}
   s^{{\rm q},{\rm L}}(x,y)
   =
   \max\!\left\{
   \hat{\rm q}_{\alpha,\,n_{t}}(x)-y,
   0
   \right\}.
\end{align}
This leads to a corresponding $1-\alpha$ {\it lower-bounded CP interval}
$\mathcal C^{\rm q,L}$ for a new feature $X_{n+1}$
\begin{equation}
\mathcal C^{\rm q,L}_{n+1,\alpha}
:=
\left\{
y\in\mathbb R:
s^{{\rm q},{\rm L}}(X_{n+1},y)
\le
Q^{\rm q,L}_{\alpha,n_{c}}
\right\},
\label{eq:onesided_set_def_general_max}
\end{equation}
where $Q^{\rm q,L}_{\alpha,n_{c}}$ denotes the empirical
$(1-\alpha)(1+1/n_{c})$-quantile of the calibration scores
$\{S_i^{\rm q,L},\, i\in\mathcal D_{n_{c}}\}$.
Since the conformity scores are non-negative, we necessarily have
$Q^{\rm q,L}_{\alpha,n_{c}}\ge 0$. This implies that the defining inequality in Eq.~\eqref{eq:onesided_set_def_general_max} can be rearranged as
\[
%s^{{\rm q},{\rm L}}(X_{n+1},y) \le Q^{\rm q,L}_{\alpha,n_{c}} \quad\Longleftrightarrow\quad 
\max\!\left\{
\hat q_{\alpha,n_{t}}(X_{n+1})-y,
0
\right\}
\le
Q^{\rm q,L}_{\alpha,n_{c}}\quad\Longleftrightarrow\quad
\begin{cases}
    y \ge \hat q_{\alpha,n_{t}}(X_{n+1}) - Q^{\rm q,L}_{\alpha,n_{c}} & \text{if}\ \hat q_{\alpha,n_{t}}(X_{n+1}) > y \\ 
    y \ge -\infty & \text{if}\ \hat q_{\alpha,n_{t}}(X_{n+1}) \leq y
\end{cases}.
\]
Taken jointly, the resulting prediction set admits the explicit representation
\begin{equation*}
\mathcal C^{\rm q,L}_{n+1,\alpha}
=
[L_{n+1}^{\rm q},\infty),
\qquad\text{with}\qquad
L_{n+1}^{\rm q}
:= \hat q_{\alpha,n_{t}}(X_{n+1}) -
Q^{\rm q,L}_{\alpha,n_{c}}.
\end{equation*}

\begin{comment}
\begin{equation*}
\mathcal C^{\rm q,L}_{n+1,\alpha}
=
[L_{n+1}^{\rm q},\infty),
\qquad\text{with}\qquad
L_{n+1}^{\rm q}
:=
\begin{cases}
\hat q_{\alpha,n_{t}}(X_{n+1}) -
Q^{\rm q,L}_{\alpha,n_{c}} & \text{if}\ \hat q_{\alpha,n_{t}}(X_{n+1}) > y \\ 
- \infty & \text{if}\ \hat q_{\alpha,n_{t}}(X_{n+1}) \leq y
\end{cases}.
\end{equation*}
\end{comment}
The $1-\alpha$ {\it upper-bounded CP interval} is obtained in a similar fashion.

\begin{remark}
One-sided CQR can be interpreted as a degenerate case of standard two-sided CQR, with one of the two quantile predictors set to $\pm \infty$, therefore neglecting one of the two sides. The two-sided CQR score can be then re-written as
\begin{align*}
s^{\rm q}(x,y)&=\max\{s^{\rm q,L}(x,y),\,s^{\rm q,U}(x,y)\}\\ &= \max\{\max\!\left\{
   \hat{q}_{\alpha,\,n_{t}}(x)-y,
   0
   \right\},\,\max\!\left\{
   y-\hat{q}_{1-\alpha,\,n_{t}}(x),
   0
   \right\}\}\\
   & = \begin{cases}
       \max\{0,\,\max\!\left\{
   y-\hat{q}_{1-\alpha,\,n_{t}}(x),
   0
   \right\}\} = s^{\rm q,U}& \text{if}\ \hat{q}_{\alpha,\,n_{t}}(x) \to -\infty\\
   \max\{\max\!\left\{
   \hat{q}_{\alpha,\,n_{t}}(x)-y,
   0
   \right\},\,0\} =  s^{\rm q,L} & \text{if}\ \hat{q}_{1-\alpha,\,n_{t}}(x) \to \infty
   \end{cases}.
\end{align*}
%accounts for both upper and lower quantile scores to control the coverage globally on both tails, one-sided CQR retains only one of the two components, corresponding %in general (for small $\alpha$ and non-degenerate scores) 
%to either $s^{q,L} = \max\{s^{q,L}, 0\}$ or $s^{q,U} = \max\{0,\,s^{q,U}\}$
 %This translates into intervals 
%using $s^{q,L}$ alone ignores upper-tail violations, whereas using $s^{q,U}$ alone ignores lower-tail violations. Hence, 
%with the omitted side becoming unconstrained, so the prediction set extends to $+\infty$ or $-\infty$, yielding sets 
%of the form $[L_{n+1}^{q},+\infty)$ or $(-\infty,U_{n+1}^{q}]$, respectively.
\end{remark}

As an alternative version of the quantile score in Eq.~\eqref{eq: cqr_onesided}, one could consider $s^{{\rm q\text{-}sgn},{\rm L}}(x,y) = \hat{q}_{\alpha,\,n_{t}}(x)-y$ and $s^{{\rm q\text{-}sgn},{\rm U}}(x,y) = y - \hat{q}_{1-\alpha,\,n_{t}}(x)$, which are still valid conformity scores, but lead to more efficient prediction intervals, as later discussed in the theoretical and empirical results. 
In Table~\ref{tab:onesided_cp_scores}, we report the corresponding lower- and upper-bounded CP intervals for this variant, along with the one-sided residual-based scores of interest defined in Table~\ref{tab:cp_scores}. %Since the one-sided quantile score has been defined in Eq.~\eqref{eq: cqr_onesided}, in Table except for the quantile score, for which ~\ref{tab:onesided_cp_scores}, we only report the non-degenerate variant.

\begin{table}[htbp]
\centering
\caption{One-sided conformity scores and associated lower/upper bounded split CP intervals.}
\label{tab:onesided_cp_scores}
\small
\renewcommand{\arraystretch}{1.2}
\setlength{\tabcolsep}{4pt}

\begin{tabularx}{\textwidth}{
>{\raggedright\arraybackslash}p{0.155\textwidth}
>{\centering\arraybackslash}p{0.03\textwidth}
>{\centering\arraybackslash}X
>{\centering\arraybackslash}p{0.45\textwidth}}
\toprule

\textbf{Score name} 
& \textbf{Tail} 
& \textbf{Score definition} 
& \textbf{One-sided split CP interval} \\

\midrule

Residual $s^{\rm res}$
& {\rm L}
& $s^{{\rm res},{\rm L}}(x,y)=\hat f_{n_{t}}(x)-y$
& $\mathcal C^{\rm res,L}_{\alpha}(x)
=\bigl[\hat f_{n_{t}}(x)-Q^{\rm res,L}_{\alpha,n_{c}},+\infty\bigr)$
\\

& {\rm U}
& $s^{{\rm res},{\rm U}}(x,y)=y-\hat f_{n_{t}}(x)$
& $\mathcal C^{\rm res,U}_{\alpha}(x)
=\bigl(-\infty,\hat f_{n_{t}}(x)+Q^{\rm res,U}_{\alpha,n_{c}}\bigr]$
\\[6pt]

S-residual $s^{\rm s\text{-}res}$
& {\rm L}
& $s^{{\rm s\text{-}res},{\rm L}}(x,y)
=\dfrac{\hat f_{n_{t}}(x)-y}{\hat\sigma_{n_{t}}(x)}$
& $\mathcal C^{\rm s\text{-}res,L}_{\alpha}(x)
=\bigl[\hat f_{n_{t}}(x)
-\hat\sigma_{n_{t}}(x)Q^{\rm s\text{-}res,L}_{\alpha,n_{c}},
+\infty\bigr)$
\\

& {\rm U}
& $s^{{\rm s\text{-}res},{\rm U}}(x,y)
=\dfrac{y-\hat f_{n_{t}}(x)}{\hat\sigma_{n_{t}}(x)}$
& $\mathcal C^{\rm s\text{-}res,U}_{\alpha}(x)
=\bigl(-\infty,
\hat f_{n_{t}}(x)
+\hat\sigma_{n_{t}}(x)Q^{\rm s\text{-}res,U}_{\alpha,n_{c}}\bigr]$
\\[6pt]

Quantile $s^{\rm q\text{-}sgn}$
& {\rm L}
& $s^{{\rm q\text{-}sgn},{\rm L}}(x,y)
=\hat q_{\alpha,n_{t}}(x)-y$
& $\mathcal C^{\rm q\text{-}sgn,L}_{\alpha}(x)
=\bigl[\hat q_{\alpha,n_{t}}(x)-Q^{\rm q\text{-}sgn,L}_{\alpha,n_{c}},
+\infty\bigr)$
\\

& {\rm U}
& $s^{{\rm q\text{-}sgn},{\rm U}}(x,y)
=y-\hat q_{1-\alpha,n_{t}}(x)$
& $\mathcal C^{\rm q\text{-}sgn,U}_{\alpha}(x)
=\bigl(-\infty,
\hat q_{1-\alpha,n_{t}}(x)
+Q^{\rm q\text{-}sgn,U}_{\alpha,n_{c}}\bigr]$
\\

\bottomrule
\end{tabularx}
\end{table}

The ACI framework for non-exchangeable data can be extended naturally to the one-sided setting, by simply applying the
ACI update separately to each tail. Specifically, define now the directional miscoverage indicators $\mathrm{err}_{i}^- := \mathbf{1}\{Y_{i} \notin \mathcal{C}^L_{i,\alpha^-_i}\}$ and $\mathrm{err}_{i}^+ := \mathbf{1}\{Y_{i} \notin \mathcal{C}^U_{i,\alpha^+_i}\}$ for $i\geq n+1$, where $\alpha^+$ and $\alpha^-$ denote the target violation rate for each tail, set e.g., at $\alpha/2$ in order to recover a global miscoverage level
$\alpha$. Taken compactly, the adaptive miscoverage levels are then updated via two independent ACI
recursions as
\begin{equation*}
\alpha_{i+1}^\pm
= \alpha_{i}^\pm + \gamma^\pm\bigl(\alpha^\pm - \mathrm{err}_{i}^\pm\bigr),\quad i\geq n+1.
\label{eq:aci_onesided}
\end{equation*}

Similarly, DtACI is extended by maintaining separate collections of candidate lower/upper
 learning rates $\{\gamma_j^\pm\}_{j=1}^k$, with corresponding sequences of adaptive miscoverage levels $\{\alpha_i^{j,\pm}\}_{j=1}^k$, for $i\geq n+1$. 
 
 \begin{comment}
     
 Here, tail-specific weights $w_i^{j,\pm}$ are obtained based on the pinball loss
$\ell(\beta_i^\pm,\alpha_i^{j,\pm})$, where $\beta_i^\pm:=
\sup\Bigl\{\beta^\pm\in[0,1]:\, Y_i\in \mathcal{C}_{i,\beta^\pm}^{\pm}\Bigr\}$ denotes the largest
miscoverage level such that $Y_i$ is contained in the corresponding one-sided
prediction set.\\
As in the two-sided case, the weights are normalized to obtain probability
distributions
\[
p_i^{j,\pm} := \frac{w_i^{j,\pm}}{\sum_{l=1}^k w_i^{l,\pm}},
\]
which define mixtures over the candidate predictors. Accordingly, we consider
the aggregated adaptive miscoverage levels
\[
\bar{\alpha}_i^\pm := \sum_{l=1}^k p_i^{l,\pm}\,\alpha_i^{l,\pm},
\]
which are used to construct the one-sided prediction sets.

Each candidate is then updated independently via the one-sided analogue of the
ACI recursion,
\[
\alpha_{i+1}^{j,\pm}
=
\alpha_i^{j,\pm}
+
\gamma_j^\pm\bigl(\alpha^\pm - \mathrm{err}_i^{j,\pm}\bigr), \quad i \geq n+1
\]
with $\mathrm{err}_i^{j,\pm}$ defined analogously to the ACI case.
\end{comment}

This decoupled construction highlights an important structural property: the control of lower or upper tail miscoverage is treated through two independent online calibration problems. As a result, the method inherits both the adaptive nature of conformal inference and the regret guarantees of DtACI, while also ensuring long-run directional control of tail-specific miscoverage rates even under nonstationarity. This tail-specific control is particularly advantageous in asymmetric and heavy-tailed settings, where the lower and upper tails %experience markedly different pressures induced by the data, and therefore 
may benefit from a targeted adaptation mechanism.

%This formulation reveals that the one-sided DtACI procedure can be interpreted as running two independent online learning problems—one for each tail—each of which adaptively combines a family of ACI experts. 

%\paragraph*{Two-sided intervals via intersection of one-sided conformal sets}

%Standard conformal prediction constructs two-sided prediction intervals by directly calibrating a single miscoverage level $\alpha$, yielding intervals of the form $\mathcal{C}_{i,\alpha}$ that guarantee marginal coverage $\mathbb{P}\{Y_i \in \mathcal{C}_{i,\alpha}\} \ge 1-\alpha$.

%While this ensures valid global marginal coverage, it does not provide any control on the distribution of errors across the lower and upper tails. In particular, the resulting procedure may exhibit asymmetric miscoverage, with violations concentrated disproportionately in one tail.

This property motivates a natural construction of tail-calibrated two-sided CP intervals by
combining the two independently built one-sided CP intervals.
Let $\mathcal C^L_{n+1,\alpha^-}$ and $\mathcal C^U_{n+1,\alpha^+}$ denote the lower and upper CP intervals for the response $Y_{n+1}$ at the target levels $\alpha^-$ and $\alpha^+$, respectively. Then, defining $\underline{\alpha} := (\alpha^-, \alpha^+)$, the two-sided prediction interval can be derived as the intersection
\begin{equation}
\mathcal C_{n+1,\underline{\alpha}}^{\cap}
: =
\mathcal C_{n+1,\alpha^-}^{L}
\;\cap\;
\mathcal C_{n+1,\alpha^+}^{U} =[L_{n+1},U_{n+1}].
\label{eq:intersection_interval}
\end{equation}

In the non-exchangeable adaptive CP setting, the construction
extends through the adaptive lower and upper miscoverage levels
$\alpha_{n+1}^-$ and $\alpha_{n+1}^+$. Defining $\underline{\alpha}_{n+1}
:=
(\alpha_{n+1}^-,\alpha_{n+1}^+)$, the resulting intersection interval is given by $\mathcal C_{n+1,\underline{\alpha}_{n+1}}^{\cap}
:=
\mathcal C_{n+1,\alpha_{n+1}^-}^{L}
\cap
\mathcal C_{n+1,\alpha_{n+1}^+}^{U}$. The full procedure is given in the Supplementary material B (Algorithm 1) for DtACI. ACI and the non-adaptive construction (for exchangeable data) can be recovered as a special case, by fixing $\gamma^\pm_j = \gamma^\pm$ and $\gamma^\pm_j = \gamma^\pm$ as well as $\alpha_i^\pm = \alpha^\pm$, respectively.

Two-sided CP intervals constructed by intersection simultaneously satisfy: (a) global marginal coverage at target level $1-\alpha$, with $\alpha=\alpha^+ +\alpha^-$ and
(b) tail-specific marginal coverage guarantees at the levels $1-\alpha^-$ and $1-\alpha^+$, as proved in Section~\ref{sec: theory}.

\section{Theoretical Results} \label{sec: theory}
We now establish finite-sample coverage guarantees under exchangeability and asymptotic coverage guarantees under non-exchangeability along with short-term regret bounds, both at the global marginal level (Theorems) and tail-specific level (Propositions).

\subsection*{Exchangeable case} \label{sec: theory_exch}

\begin{proposition}[Tail-specific coverage guarantees]
\label{thm:onesided_validity_three_families}
Under exchangeability, the lower/upper one-sided split CP intervals in Table~\ref{tab:onesided_cp_scores} satisfy tail-specific $1-\alpha^-$ / $1-\alpha^+$ coverage guarantees:
\begin{equation}
\mathbb P\!\left(Y_{n+1}\in\mathcal C_{n+1,\alpha^-}^L\right)\ge 1-\alpha^-,
\qquad
\mathbb P\!\left(Y_{n+1}\in\mathcal C_{n+1,\alpha^+}^{U}\right)\ge 1-\alpha^+.
\label{eq:onesided_lower_bounds_general}
\end{equation}
Moreover, if the corresponding scores are almost surely distinct, then
\begin{equation}
\mathbb P\!\left(Y_{n+1}\in \mathcal C_{n+1,\alpha^-}^L\right)\le 1-\alpha^-+\frac{1}{n_{c}+1},
\quad 
\mathbb P\!\left(Y_{n+1}\in\mathcal C_{n+1,\alpha^+}^{U}\right)\le 1-\alpha^+ + \frac{1}{n_{c}+1}.
\label{eq:onesided_upper_bounds_general}
\end{equation}
\end{proposition}

\begin{theorem}[Global coverage guarantees]
\label{thm:two_sided_coverage}
Assume all conformity score functions are quasi-convex in $y$. Let $L_{n+1}$ and $U_{n+1}$ be the lower and upper one-sided CP interval bounds for $n+1$, with miscoverage levels $\alpha^{\pm}\in(0,1)$. Let $\underline{\alpha}
:=
(\alpha^-,\alpha^+)$, and define %the associated CP interval obtained as the intersection of the two one-sided CP intervals, that is,
\[
\mathcal C_{n+1,\underline{\alpha}}^{\cap}
:=
[L_{n+1},U_{n+1}]
=
\mathcal C_{n+1,\alpha^-}^{L}
\cap
\mathcal C_{n+1,\alpha^+}^{U}.
\]
Under exchangeability, if the one-sided marginal guarantees  in Eq.~\eqref{eq:onesided_lower_bounds_general} hold, then
\begin{equation*}
1-(\alpha^{-}+\alpha^{+})
\;\le\;
\mathbb P\!\left(Y_{n+1}\in \mathcal C_{n+1,\underline{\alpha}}^{\cap}\right).
\end{equation*}
Moreover, if the corresponding scores are almost surely distinct, so that \eqref{eq:onesided_upper_bounds_general} holds, then
\begin{equation*}
1-(\alpha^{-}+\alpha^{+})
\;\le\;
\mathbb P\!\left(Y_{n+1}\in \mathcal C_{n+1,\underline{\alpha}}^{\cap}\right)
\;\le\;
1-\max\{\alpha^{-},\alpha^{+}\}
+\frac{1}{n_{c}+1}.
\end{equation*}
In particular, if $\alpha^{-}=\alpha^{+}=\alpha/2$, then $1-\alpha
\;\le\;
\mathbb P\!\left(Y_{n+1}\in \mathcal C_{n+1,\underline{\alpha}}^{\cap}\right)
\;\le\;
1-\frac{\alpha}{2}
+\frac{1}{n_{c}+1}$.
\end{theorem}

%Setting $\alpha^-=\alpha^+=\alpha/2$ does not recover the exact $1-\alpha$ finite-sample guarantee of 
Compared to classical two-sided CP intervals, the intersection proposal is characterized by a potential efficiency loss due to the non-vanishing gap of value $\alpha/2$ between the upper and lower probability bounds as $n_c \to \infty$. Proof of Theorem \ref{thm:two_sided_coverage} is provided in Supplementary material C, while that of Proposition \ref{thm:onesided_validity_three_families} follows verbatim the theorems in \cite{lei2018distribution}.%Thus, the intersection construction should be viewed as complementary to classical CP: two-sided scores are preferable when tight global guarantees are required, whereas one-sided intervals are more appropriate when tail-specific coverage control is the main objective.
%Note that setting $\alpha^{-}=\alpha^{+}=\alpha/2$ does not recover the exact $1-\alpha$ coverage guarantee of classical conformal intervals: even as $n_{c}\to\infty$, a non-vanishing gap persists between the lower and upper bounds on the two-sided coverage probability. From a practical standpoint, the intersection of one-sided conformal intervals should therefore be viewed as a complementary tool rather than a replacement for classical conformal intervals. When tight \textit{global} finite-sample guarantees are critical, joint calibration via the two-sided scores remains preferable. Conversely, when one-sided coverage control is the primary objective, the one-sided construction provides a principled and interpretable alternative. Moreover, another key practical advantage of the one-sided construction is its intrinsic adaptivity. Unlike symmetric conformal intervals, which inflate both endpoints by the same amount irrespective of where violations occur, the one-sided approach adjusts each bound independently, calibrating it solely based on its own empirical violation history. As a consequence, each endpoint is expanded only to the extent required to control miscoverage in the corresponding tail, leading to potentially tighter and more efficient prediction intervals, particularly in settings with asymmetric distributions and strong serial dependence, including highly autocorrelated time series.

\subsection*{Non-exchangeable case} \label{sec: theory_nonexch}

For notational convenience, results below are stated with the online sequence re-indexed from $1$, where index $1$ corresponds to time $n+1$ after the first $n$ units have been observed. We denote by $N$ the total length of the online sequence. Proofs of all the theoretical results stated here, along with any required Lemma, are provided in Supplementary material C.

% and by $k$ the number of candidate learning rates in the DtACI grid $\{\gamma_j\}_{j=1}^k$, as introduced in Section~\ref{sec: split-CP}.

%\subsubsection*{Distribution-free guarantees for one-sided ACI}
%For notational simplicity, we re-index the online sequence from $1$, with index $1$ corresponding to the original time $n+1$, after the first $n$ observations have been used for initialization. We derive the coverage guarantees of the one-sided adaptive conformal procedures under the same assumptions as in \citep{gibbs2021adaptive}: $\alpha_1^\pm\in[0,1]$, and the empirical quantile functions $\widehat Q_{\alpha,n_c}^{L,U}$ are non-decreasing, with boundary values $\widehat Q_{\alpha,n_{c}}^{L,U}(x) = -\infty$ for all $x < 0$, and $\widehat Q_{\alpha,n_{c}}^{L,U}(x) = +\infty$ for all $x > 1$. Under these conditions, adaptive conformal inference achieves the target long-run coverage frequency without distributional assumptions on the data-generating process.

\begin{proposition}[Tail-specific coverage guarantees under ACI]
\label{prop:aci_lower}
Under the same conditions as \citet{gibbs2021adaptive}, namely, that the empirical quantile functions $\widehat{Q}_{\alpha,n_{c}}^{L,U}$ are non-decreasing with $\widehat{Q}_{\alpha,n_{c}}^{L,U}(x)=-\infty$ for $x<0$ and $\widehat{Q}_{\alpha,n_{c}}^{L,U}(x)=+\infty$ for $x>1$, for a given $\alpha_1^\pm\in[0,1]$, with probability one, for all $N \in \mathbb N$,
\[
\left|
\frac{1}{N}\sum_{i=1}^N \mathrm{err}_i^{\pm}
-
\alpha^{\pm}
\right|
\le
\frac{\max\{\alpha_1^{\pm},\,1-\alpha_1^{\pm}\} + \gamma}{N\gamma}.
\]
In particular, $\frac{1}{N}\sum_{i=1}^N \mathrm{err}_i^{\pm}
\;\xrightarrow[N\to\infty]{\mathrm{a.s.}}\;
\alpha^{\pm}$.
\end{proposition}

\begin{theorem}[Global coverage guarantees under ACI]
\label{thr:aci_total}
Under the same conditions as Proposition~\ref{prop:aci_lower}, let
$\mathrm{err}_i^{\mathrm{tot}} := \mathrm{err}_i^- + \mathrm{err}_i^+$
and $\alpha^{\mathrm{tot}} := \alpha^- + \alpha^+$ denote the observed and target miscoverage levels of the intersection interval $\mathcal{C}_{i,\underline{\alpha}_i}^{\cap}$, obtained from two independent ACI procedures run with the 
%Consider two one-sided ACI procedures run in parallel, one for the lower tail and one for the upper tail. Define $ \mathrm{err}_i^- := \mathbf 1\{Y_i \notin \mathcal C_{i,\alpha_i^-}^{L}\}$ and $ \mathrm{err}_i^+ := \mathbf 1\{Y_i \notin \mathcal C_{i,\alpha_i^+}^{U}\}$, where $\alpha_i^-$ and $\alpha_i^+$ are the adaptive lower- and upper-tail miscoverage levels, respectively, with target levels $\alpha^-$ and $\alpha^+$.
%Assume that both ACI procedures use the 
same step size $\gamma>0$, so that $\alpha_i^\pm \in [-\gamma, 1+\gamma]$ almost surely for all $i$ (Lemma~3.1 in Supplementary Material). Then, with probability one, for every $N\in\mathbb N$,
\[
\left|
\frac{1}{N}\sum_{i=1}^N \mathrm{err}_i^{\mathrm{tot}}
-
\alpha^{\mathrm{tot}}
\right|
\le
\frac{
\max\{\alpha_1^-+\alpha_1^+,\;2-(\alpha_1^-+\alpha_1^+)\}+2\gamma
}{
N\gamma
}.
\]
where $\mathrm{err}_i^{\mathrm{tot}}
:=
\mathrm{err}_i^-+\mathrm{err}_i^+$ and 
$\alpha^{\mathrm{tot}}
:=
\alpha^-+\alpha^+$. In particular, $\frac{1}{N}\sum_{i=1}^N \mathrm{err}_i^{\mathrm{tot}}
\;\xrightarrow[N\to\infty]{\mathrm{a.s.}}\;
\alpha^{\mathrm{tot}}$.
\end{theorem}

%\subsubsection*{Distribution-free guarantees for one-sided DtACI} All the results developed for the DtACI procedure in \citep{gibbs2024conformal} extend directly to the one-sided setting after replacing $\beta_i$ with $\beta_i^\pm$ and $\alpha_i$ with $\alpha_i^\pm$. Below, we restate the main results in their one-sided formulation for the lower tail; completely analogous statements hold for the upper tail. The corresponding proofs follow identically from those in \citep{gibbs2024conformal} and are therefore omitted.

% -------------------------------------------------------------------------
%\paragraph*{Bounds on the long-term one-side coverage} The preceding results establish local coverage control for DtACI. We now consider the asymptotic regime and show that, when $\eta_i^\pm\to 0$ and $\sigma_i^\pm\to 0$ as $i\to\infty$, the long-run average coverage converges to the nominal level $1-\alpha$. 
\begin{proposition}[Tail-specific coverage guarantees under DtACI]
\label{prop:dtaci_onesided_longterm}
Let $\gamma^-_{\min} := \min_j\gamma^-_j$ and $\gamma^-_{\max} := \max_j\gamma^-_j$.
Under the lower-tail DtACI procedure with time-varying parameters $\eta^-_i$
and $\sigma^-_i$, with probability one, for every $N \in \mathbb{N}$,
\[
\left|
\frac{1}{N}\sum_{i=1}^N \mathbb E[\mathrm{err}_i^-]-\alpha^-
\right|
\;\le\;
\frac{1+2\gamma^-_{\max}}{N\gamma^-_{\min}}
+
\frac{(1+2\gamma^-_{\max})^2}{\gamma^-_{\min}}
\cdot \frac1N\sum_{i=1}^N \eta^-_i e^{\eta^-_i(1+2\gamma^-_{\max})}
+
2\frac{1+\gamma^-_{\max}}{\gamma^-_{\min}}\cdot \frac1N\sum_{i=1}^N \sigma^-_i.
\]
where the expectation is over the DtACI aggregation probabilities
$\{p_i^{j,-}\}_{j=1}^k$. In particular, if $\eta_i^{-}\to 0$ and $\sigma_i^{-}\to 0$ as $i\to\infty$, then
$\frac{1}{N}\sum_{i=1}^{N}\mathrm{err}_i^{-}
\xrightarrow[N\to\infty]{\mathrm{a.s.}}\alpha^-$.
\end{proposition}

%The following theorem provides bounds on the short-term coverage behavior of the intersection interval.

%The following theorem provides bounds on the asymptotic coverage of the intersection interval.

\begin{theorem}[Global coverage under DtACI]
\label{thm:dtaci_total_common_gamma}
Let $\mathrm{err}_i^{\mathrm{tot}} := \mathrm{err}_i^- + \mathrm{err}_i^+$
and $\alpha^{\mathrm{tot}} := \alpha^- + \alpha^+$ denote the observed and target miscoverage level of the intersection interval $\mathcal{C}_{i,\underline{\alpha}_i}^{\cap}$, obtained from two independent DtACI procedures run with the same candidate
grid $\{\gamma_j\}_{j=1}^k$ with $\gamma_{\min} := \min_j \gamma_j$ and
$\gamma_{\max} := \max_j \gamma_j$, but possibly different learning-rate
sequences $\{\eta_i^{\pm}\}_{i \ge 1}$ and mixing sequences
$\{\sigma_i^{\pm}\}_{i \ge 1}$. Assume the conformity score functions are quasi-convex
in $y$. Given target levels $\alpha^-, \alpha^+ \in (0,1)$, define
\[
\tilde{\alpha}_i^{\pm} := \sum_{l=1}^k \frac{p_i^{l,\pm}\alpha_i^{l,\pm}}{\gamma_l},
\qquad
\tilde{\alpha}_i^{\mathrm{tot}} := \tilde{\alpha}_i^- + \tilde{\alpha}_i^+,
\qquad
B_0 := \max\!\left\{
  \tilde{\alpha}_1^{\mathrm{tot}} + \frac{2\gamma_{\max}}{\gamma_{\min}},\;
  \frac{2+2\gamma_{\max}}{\gamma_{\min}} - \tilde{\alpha}_1^{\mathrm{tot}}
\right\}.
\]
Then, for every $N\in\mathbb N$,
\begin{align}
\label{eq:dtaci_total_common_gamma_bound}
\left|
  \frac{1}{N}\sum_{i=1}^{N} \mathbb{E}[\mathrm{err}_i^{\mathrm{tot}}]
  - \alpha^{\mathrm{tot}}
\right|
\;\le\;
\frac{B_0}{N}
&+
\frac{(1+2\gamma_{\max})^2}{\gamma_{\min}}
\cdot\frac{1}{N}\sum_{i=1}^{N}
\Bigl[
  \eta_i^{-}e^{\eta_i^{-}(1+2\gamma_{\max})}
  +
  \eta_i^{+}e^{\eta_i^{+}(1+2\gamma_{\max})}
\Bigr] \nonumber \\
&+
2\,\frac{1+\gamma_{\max}}{\gamma_{\min}}
\cdot\frac{1}{N}\sum_{i=1}^{N}(\sigma_i^{-}+\sigma_i^{+}),
\end{align}
where the expectation is over the DtACI aggregation probabilities
$\{p_i^{j,\pm}\}_{j=1}^k$. In particular, if $\eta_i^{\pm}\to 0$ and $\sigma_i^{\pm}\to 0$ as
$i\to\infty$, then
$\frac{1}{N}\sum_{i=1}^{N}\mathrm{err}_i^{\mathrm{tot}}
\xrightarrow[N\to\infty]{\mathrm{a.s.}}\alpha^{\mathrm{tot}}$.\\ 
When $\eta_i^- = \eta_i^+ =: \eta_i$ and $\sigma_i^- = \sigma_i^+ =: \sigma_i$
for all $i$, Eq.~\eqref{eq:dtaci_total_common_gamma_bound} simplifies to
\[
\left|
  \frac{1}{N}\sum_{i=1}^{N} \mathbb{E}[\mathrm{err}_i^{\mathrm{tot}}]
  - \alpha^{\mathrm{tot}}
\right|
\;\le\;
\frac{B_0}{N}
+
2\,\frac{(1+2\gamma_{\max})^2}{\gamma_{\min}}
\frac{1}{N}\sum_{i=1}^N
\eta_i\,e^{\eta_i(1+2\gamma_{\max})}
+
4\,\frac{1+\gamma_{\max}}{\gamma_{\min}}
\frac{1}{N}\sum_{i=1}^N \sigma_i.
\]
\end{theorem}

The results established so far provide asymptotic coverage guarantees for the intersection interval under both ACI and DtACI. We now complement these with finite-sample bounds on how closely the adaptive miscoverage levels track their oracle targets, characterising the short-term calibration behaviour of the intersection interval under DtACI.

\begin{proposition}[Tail-specific short-term bounds under DtACI]
\label{prop:dtaci_onesided_regret}
Let $\{\alpha_i^{*,-}\}_{i \in I}$ be any comparator sequence of oracle miscoverage levels, representing the best adaptive strategy in hindsight over the interval $I$. Let $\gamma^-_{\max} := \max_{1 \le j \le k}\gamma^-_j$
and assume $\gamma^-_1 < \gamma^-_2 < \cdots < \gamma^-_k$ with
$\gamma^-_{j+1}/\gamma^-_j \le 2$ for all $1 < j < k$,
$\gamma^-_k \ge \sqrt{1+1/|I|}$, $\sigma = 1/(2|I|)$,
$\gamma^-_1 \le \sqrt{(\sum_{i=r+1}^s
|\alpha_i^{*,-}-\alpha_{i-1}^{*,-}|+1)/|I|}$, and
$\eta^- = \sqrt{(\log(k|I|)+2)/
\sum_{i=r}^s\mathbb{E}[\ell(\beta_i^-,\alpha_i^-)^2]}$.
Then, 
\begin{enumerate}

\item \textbf{(Dynamic regret)} For any interval $I=[r,s]\subset[N]$,
\begin{align*}
\frac{1}{|I|}\sum_{i=r}^s \mathbb{E}[\ell(\beta^-_i,\alpha^-_i)]
-
\frac{1}{|I|}\sum_{i=r}^s \ell(\beta^-_i,\alpha_i^{*,-})
&\le
2\sqrt{\frac{\log(k|I|)+2}{|I|}}
\sqrt{\frac{1}{|I|}\sum_{i=r}^s\mathbb{E}[\ell(\beta^-_i,\alpha^-_i)^2]}\\
&\quad+
4(1+\gamma^-_{\max})^2
\sqrt{\frac{\sum_{i=r+1}^s|\alpha_i^{*,-}-\alpha_{i-1}^{*,-}|+1}{|I|}}\\
&=
O\!\left(\sqrt{\frac{\log|I|}{|I|}}\right)
+
O\!\left(\sqrt{\frac{\sum_{i=r+1}^s|\alpha_i^{*,-}-\alpha_{i-1}^{*,-}|}{|I|}}\right),
\end{align*}
where the expectation is over the DtACI aggregation probabilities
$\{p_i^{j,-}\}_{j=1}^k$.

\item \textbf{(Coverage deviation)} Additionally assume $\beta^-$ has density $p^-(\cdot)$ on $[0,1]$ with $p^-(x) \ge p_- > 0$, and let $\alpha_i^{\star,-}$ satisfy $\mathbb{P}(Y_i \in \mathcal{C}^L_{i,\alpha_i^{\star,-}}
\mid \{\beta_u^-\}_{u<i}) = 1-\alpha^-$. Then, %combining part~(i) with the bound $|\alpha_i^- - \alpha_i^{\star,-}|^2 \le (2/p^-)\bigl(\mathbb{E}[\ell(\beta^-,\alpha_i^-)] - \mathbb{E}[\ell(\beta^-,\alpha_i^{\star,-})]\bigr)$,
\begin{equation}
\label{eq:onesided_mse_from_regret}
\frac{1}{|I|}\sum_{i=r}^s
\frac{p_-}{2}\,\mathbb{E}\!\bigl[(\alpha_i^--\alpha_i^{\star,-})^2\bigr]
\;\le\;
O\!\left(\sqrt{\frac{\log|I|}{|I|}}\right)
+
O\!\left(\sqrt{\frac{1}{|I|}\sum_{i=r+1}^s
\mathbb{E}|\alpha_i^{*,-}-\alpha_{i-1}^{*,-}|}\right),
\end{equation}
where the expectation is over $\{p_i^{j,-}\}_{j=1}^k$ and
$\{\beta^-_i\}_{i \le s}$.
\end{enumerate}
\end{proposition}

The variation term $\sum_{i=r+1}^s|\alpha_i^{*,\pm}-\alpha_{i-1}^{*,\pm}|/|I|$ measures the magnitude of distribution shift over $I$; since the bound holds for every interval of fixed length, DtACI responds to nonstationarity in a local-in-time manner across the full sequence.

\begin{theorem}[Global short-term bound under DtACI]
\label{thr:pinball_to_alpha_total_improved}
Let $\alpha_i^{\mathrm{tot}} := \alpha_i^- + \alpha_i^+$ and
$\alpha_i^{*,\mathrm{tot}} := \alpha_i^{*,-} + \alpha_i^{*,+}$ denote the
adaptive and oracle total miscoverage levels for the intersection interval
$\mathcal{C}_{i,\underline{\alpha}_i}^{\cap}$, obtained from two independent
DtACI procedures run in parallel under the same conditions as
Proposition~\ref{prop:dtaci_onesided_regret}, with conformity score functions
quasi-convex in $y$. Assume $\beta^\pm$ has density $p^\pm(\cdot)$ on
$[0,1]$ satisfying $p^\pm(x) \ge p_\pm > 0$, and that there exists
$\alpha^{\star,\pm}$ with $\mathbb{P}(\beta^\pm < \alpha^{\star,\pm}) =
\alpha^\pm$. Define $p_{\min} := \min\{p_-, p_+\}$ and
$G_i^\pm := \mathbb{E}[\ell(\beta^\pm,\alpha_i^\pm)] -
\mathbb{E}[\ell(\beta^\pm,\alpha_i^{*,\pm})]$.
Then, for every $i$,
\begin{equation}
\label{eq:prop_total_improved_simpler}
(\alpha_i^{\mathrm{tot}}-\alpha_i^{*,\mathrm{tot}})^2
\;\le\;
\frac{4}{p_{\min}}\bigl(G_i^-+G_i^+\bigr).
\end{equation}
Furthermore, for any interval $I=[r,s]\subset[N]$,
\begin{align}
\label{eq:short_term_total_like6_improved}
\frac{1}{|I|}\sum_{i=r}^s
\frac{p_{\min}}{4}\,
\mathbb{E}\!\left[(\alpha_i^{\mathrm{tot}}-\alpha_i^{*,\mathrm{tot}})^2
\right]
&\le
O\!\left(\sqrt{\frac{\log|I|}{|I|}}\right) \nonumber \\
&+
O\!\Bigg(
\sqrt{\frac{1}{|I|}\sum_{i=r+1}^s
\mathbb{E}|\alpha_i^{*,-}-\alpha_{i-1}^{*,-}|}
+
\sqrt{\frac{1}{|I|}\sum_{i=r+1}^s
\mathbb{E}|\alpha_i^{*,+}-\alpha_{i-1}^{*,+}|}
\Bigg),
\end{align}
where the expectation is over $\{p_i^{j,\pm}\}_{j=1}^k$ and
$\{\beta_i^\pm\}_{i \le s}$. If additionally the map $\underline{\alpha}
\mapsto \mathbb{P}(Y_i \in \mathcal{C}_{i,\underline{\alpha}}^{\cap} \mid
\{\beta_s^\pm\}_{s<i})$ is $\mathcal{L}_{\cap}$-Lipschitz in
$\alpha^{\mathrm{tot}} = \alpha^- + \alpha^+$, then
\eqref{eq:short_term_total_like6_improved} also bounds the local coverage
deviation of $\mathcal{C}_{i,\underline{\alpha}_i}^{\cap}$, a condition
reasonable whenever the conformity-score distributions and their quantile
estimates vary smoothly with the miscoverage level.
\end{theorem}

\section{Experiments}\label{sec: sim_studies}

%In this section we present a comprehensive numerical comparison of conformal prediction methods, encompassing both \emph{classical closed conformal prediction} (Classic CP) and the proposed \emph{open (one-sided) conformal prediction} framework.
%The analysis is conducted across multiple iid (hence exchangeable) data generating processes, with the goal of assessing not only global coverage properties but also directional calibration and interval efficiency.

Simulation experiments are now presented to validate the proposed approach. %Each simulated time series has length $T=3000$, and results are averaged over $R=500$ Monte Carlo replications.

\textbf{Scenarios.} We consider the following data generating processes, %chosen to reflect increasing levels of distributional complexity.
each based on a sample of size $N = 3000$, replicated for $R=500$ Monte Carlo (MC) runs. Throughout all scenarios, the location and scale parameters are fixed at $\mu = 0.5$ and $\sigma = 1$.
\vspace*{-0.1cm}
\begin{description}
\item[\textit{IID Gaussian}]
Observations generated independently as
$
Y_{i} \overset{IID}{\sim} \mathcal N(\mu,\sigma^2),\ i = 1,\dots, N
$.
This setting serves as a benchmark for light-tailed and symmetric distributions.

\item[\textit{AR(1) Gaussian}] We generate
$
Y_{i} = \phi Y_{i-1} + \epsilon_i, \epsilon_i \overset{IID}{\sim} \mathcal N(\mu,\sigma^2)
$, with $\phi = 0.9$. 
\item [\textit{IID Student-$t$}]
We generate
$
Y_{i} = \mu + \sigma \tau_i
$, where $\tau_i \overset{IID}{\sim} t_\nu$, with $\nu=5$ degrees of freedom.
This scenario introduces heavy-tailed behavior while preserving symmetry.
\item[\textit{AR(1) Student-$t$}] We generate $
Y_{i} = \phi Y_{i-1} + \tau_i, \tau_i  \overset{IID}{\sim} t_\nu$ with $\nu=5$ and $\phi = 0.9$. 
\item[\textit{IID Skewed Student-$t$}]
We generate
$
Y_{i} = \mu + \sigma S_{\lambda,i}
$, where
$S_{\lambda,i} \overset{IID}{\sim} \mathrm{Skew}\text{-}t_{\nu,\lambda}$ with $\nu=5$ and
$\lambda=-3$. This scenario considers heavy tails and pronounced negative asymmetry.
\item[\textit{AR(1) Skewed Student-$t$}] We generate $
Y_{i} = \phi Y_{i-1} + S_{\lambda,i}, \text{ where } S_{\lambda,i} \overset{IID}{\sim} \mathrm{Skew}\text{-}t_{\nu,\lambda}$ with $\nu=5$,
$\lambda=-3$ and $\phi = 0.9$.
\end{description}

\textbf{Prediction model.} For all methods, we employ an autoregressive model of order one, denoted as AR(1), as the underlying forecasting model; this is given by
%Let $\{Y_i\}_{i=1}^N$ denote the simulated series, for each observation $i$, we estimate the AR(1) model
\[
Y_{i}
=
c + \phi Y_{i-1} + \varepsilon_i,\qquad i = 1,2,\dots
\]
where $c$ and $\phi$, the unconditional mean (or drift) and the autoregressive coefficient, respectively, are estimated from the data, while $\varepsilon_i \overset{IID}{\sim} \mathcal N(0,1)$ is the error component.

The fitted model is then used to produce one-step-ahead predictive quantities. In particular, for each observation $i\geq n+1$, we obtain the conditional mean forecast $\widehat \mu_i$, together with the associated conditional standard deviation $\widehat\sigma_{i}$.
%This yields a one-step-ahead point forecast $\widehat \mu_i := \mathbb E[Y_i \mid \mathcal W_{i-1}]$, as well as a predictive standard deviation $\widehat\sigma_{i} :=\sigma[Y_i \mid \mathcal W_{i-1}]$. 
These two predictors are used in the CP procedures based on residual and standardized residual scores. For CP intervals based on the quantile score, the underlying quantile predictors coincide with those used in the benchmark model.

\textbf{CP specification and compared methods.} In the dependent case, the autoregressive AR(1) structure with
coefficient $\phi=0.9$ introduces strong temporal dependence, violating exchangeability.
Therefore, all CP procedures under dependence scenarios are implemented using ACI and DtACI. In particular, %for AR(1) data-generating processes (DGPs), 
we set $\gamma=0.005$ for ACI to enable adaptive calibration under dependence. For DtACI, we instead consider the following grid of candidate $\gamma$ values:
$\gamma \in \{0.001,\;0.002,\;0.004,\;0.008,\;0.016,\;0.032,\;0.064,\;0.128\}$.

We compare the following approaches when constructing a prediction interval for $Y_{n+1}$: 
\begin{description}
    \item[Proposed CP approach $\mathcal C_{n+1}^{\cap}$:] CP interval obtained as the one-sided intersection, computed with all four scores of interest (Table~\ref{tab:onesided_cp_scores}), and including the two versions of the quantile score.

    \item[Standard CP approach $\mathcal C_{n+1}$:] standard two-sided CP interval, computed with all four scores of interest.

    \item[Asymptotic CLT approach (Benchmark):] As a parametric benchmark, we consider an asymptotic central limit theorem (CLT) based prediction interval, obtained as: %from the same AR(1) model used in the residual- and standardized-based CP methods, defined as follows:
\begin{equation*}\label{naive} \tag{Benchmark}
  \widehat L_{n+1}
  \;=\;
  \widehat \mu_{n+1}
  \;-\;
  z_{1-\alpha/2}\,\widehat\sigma_{n+1},
  \qquad
  \widehat U_{n+1}
  \;=\;
  \widehat \mu_{n+1}
  \;+\;
  z_{1-\alpha/2}\,\widehat\sigma_{n+1},
\end{equation*}
where $z_{1-\alpha/2}$ is the $(1-\alpha/2)$-quantile of the standard
normal distribution. %This benchmark relies entirely on parametric assumptions and does not provide finite-sample coverage guarantees.\\
\end{description}
%Within the class of one-sided intersection CP procedures, we additionally distinguish between two quantile-based constructions, $\mathcal C_{n+1}^{\rm q\text{-}sgn,\cap}$ and $\mathcal C_{n+1}^{\rm q,\cap}$, which are respectively obtained using the one-sided quantile conformity scores reported in Table~\ref{tab:onesided_cp_scores} and their "truncated" counterparts introduced in \eqref{eq: cqr_onesided}.

Throughout, we set the nominal coverage level to $1-\alpha = 0.9$, with tail-specific miscoverage levels given by $\alpha^+ = \alpha^- = \alpha/2 = 0.95$.

\textbf{Performance metrics.} All methods are compared on the basis of: (1) their marginal validity or {\it global marginal coverage} ($\widehat{\mathrm{Cov}}
=
\frac{1}{RN}\sum_{r=1}^R \sum_{i=1}^N \mathbf 1\{Y_{i}^{(r)}\in \mathcal C_{i}^{(r)}\}$); (2) tail-specific coverage, defined as $\widehat{\mathrm{Cov}}_{L}
=
\frac{1}{RN}\sum_{r=1}^R \sum_{i=1}^N \mathbf 1\{Y_{i}^{(r)}\in \mathcal C_{i}^{L,(r)}\}$ and $\widehat{\mathrm{Cov}}_{U}
=
\frac{1}{RN}\sum_{r=1}^R \sum_{i=1}^N \mathbf 1\{Y_{i}^{(r)}\in \mathcal C_{i}^{U,(r)}\}$, respectively; and (3) mean width (efficiency), given by the interval's length.

\textbf{Results.} Results are reported in Tables~\ref{tab:cp_aci_iid}--\ref{tab:cp_aci_ar1} for all compared methods, in both the exchangeable and non-exchangeable AR(1) settings. Due to a high similarity and for the sake of space, we only report results obtained with the ACI framework, deferring those obtained with the DtACI in Supplementary Material D. Those methods exhibit coverage behavior and efficiency patterns similar to those of the ACI method. 

Overall, all CP methods achieve satisfactory global coverage, with slight undercoverage of the naive benchmark in Gaussian settings. The results reveal pronounced asymmetries in tail-specific coverage for standard CP intervals, as well as the benchmark, under the skewed scenario. Specifically, we note undercoverage over the lower tail and overcoverage over the upper tail, reflecting sensitivity to negative skewness. In contrast, the proposed one-sided intersection approach overcomes this limitation, by calibrating each tail separately, but comes with a slightly decreased efficiency (slightly wider intervals). This is consistent with the non-vanishing gap between the upper and lower probability bounds established in Theorem~\ref{_thm:two_sided_coverage}. The residual $s^{\rm res}$ and the quantile $s^{\rm q\text{-}sgn}$ scores consistently yield the narrowest intervals. An important difference emerges between the intervals
$\mathcal C_{n+1,\underline{\alpha}}^{\rm q\text{-}sgn,\cap}$ and
$\mathcal C_{n+1,\underline{\alpha}}^{\rm q,\cap}$ in the skewed Student-$t$ setting: while the latter achieves satisfactory lower-tail coverage, it exhibits extreme upper-tail overcoverage, reaching a value equal to one. %This behavior is explained by the conformity scores \eqref{eq: cqr_onesided} used to construct $\mathcal C_{n+1,\underline{\alpha}}^{q\text{-}max,\cap}$. Indeed these scores generate ties with positive probability and therefore violate the almost-sure distinctness assumption required for the finite-sample upper coverage bounds in Theorem~\ref{_thm:two_sided_coverage}. As a consequence, under strong negative skewness, the resulting upper prediction set become highly conservative.
The extreme upper-tail overcoverage of $\mathcal{C}_{n+1,\underline{\alpha}}^{\rm q,\cap}$ follows directly from the truncated scores in Eq.~\eqref{eq: cqr_onesided}: under strong
negative skewness, many calibration observations fall below $\hat{q}_{1-\alpha,n_{t}}(x)$,
producing ties at zero in the upper-tail score. This violates the almost-sure distinctness
assumption of Theorem~\ref{_thm:two_sided_coverage} and, because the truncation prevents the
calibration quantile from going negative, the upper bound cannot move sufficiently downward,
yielding coverage equal to one. The non-truncated score $s^{{\rm q\text{-}sgn},{\rm U}}$ does not suffer
from this restriction, and can be therefore considered a better choice in practice. Further details are provided in Supplementary material D.

%=========================================================
% TABLE 1 — IID - ACI
%=========================================================
\begin{table}[htbp]
\centering
\caption{
Empirical coverage and width of compared methods (exchangeable setting), reported as mean $\pm$ standard deviation over $R = 500$ MC runs. All CP intervals are constructed using a standard split CP. Global miscoverage is set at $\alpha = 0.1$, with tail-specific levels to $\alpha/2$. Bold font flags global undercoverage, tail undercoverage, or extreme tail overcoverage.
}
\label{tab:cp_aci_iid}
\scriptsize
\setlength{\tabcolsep}{2.7pt}
\renewcommand{\arraystretch}{1.1}

\begin{tabular}{llcccccc}
\hline
\textbf{Scenario} & \textbf{Method}
& $\widehat{\mathrm{Cov}}$
& $\widehat{\mathrm{Cov}}_L$
& $\widehat{\mathrm{Cov}}_U$
& Mean width
& Median width \\
\hline

%=========================================================
% Normal IID
%=========================================================
Normal IID
& Benchmark
& $\bm{0.898 \pm 0.006}$ & $\bm{0.949 \pm 0.004}$ & $\bm{0.949 \pm 0.004}$
& $3.279 \pm 0.060$ & $3.283 \pm 0.052$ \\
& $\mathcal C_{n+1,\alpha}^{\rm res}$
& $0.903 \pm 0.006$ & $0.952 \pm 0.004$ & $0.952 \pm 0.004$
& $3.337 \pm 0.075$ & $3.307 \pm 0.063$ \\
& $\mathcal C_{n+1,\underline{\alpha}}^{\rm res,\cap}$
& $0.904 \pm 0.006$ & $0.952 \pm 0.004$ & $0.952 \pm 0.004$
& $3.353 \pm 0.075$ & $3.314 \pm 0.062$ \\
& $\mathcal C_{n+1,\alpha}^{\rm s\text{-}res}$
& $0.908 \pm 0.005$ & $0.954 \pm 0.004$ & $0.954 \pm 0.004$
& $3.438 \pm 0.098$ & $3.338 \pm 0.076$ \\
& $\mathcal C_{n+1,\underline{\alpha}}^{\rm s\text{-}res,\cap}$
& $0.909 \pm 0.006$ & $0.954 \pm 0.004$ & $0.954 \pm 0.004$
& $3.483 \pm 0.152$ & $3.345 \pm 0.076$ \\
& $\mathcal C_{n+1,\alpha}^{\rm q}$
& $0.907 \pm 0.005$ & $0.954 \pm 0.004$ & $0.954 \pm 0.004$
& $3.390 \pm 0.072$ & $3.335 \pm 0.074$ \\
& $\mathcal C_{n+1,\underline{\alpha}}^{\rm q,\cap}$
& $0.910 \pm 0.005$ & $0.955 \pm 0.003$ & $0.955 \pm 0.003$
& $3.428 \pm 0.067$ & $3.362 \pm 0.065$ \\
& $\mathcal C_{n+1,\underline{\alpha}}^{\rm q\text{-}sgn,\cap}$
& $0.908 \pm 0.005$ & $0.954 \pm 0.004$ & $0.954 \pm 0.004$
& $3.405 \pm 0.072$ & $3.342 \pm 0.075$ \\[4pt]
%=========================================================
% Student IID
%=========================================================
Student-$t$ IID
& Benchmark
& $0.910 \pm 0.007$ & $0.955 \pm 0.004$ & $0.955 \pm 0.004$
& $4.226 \pm 0.148$ & $4.231 \pm 0.133$ \\
& $\mathcal C_{n+1,\alpha}^{\rm res}$
& $0.903 \pm 0.006$ & $0.952 \pm 0.004$ & $0.951 \pm 0.004$
& $4.116 \pm 0.135$ & $4.070 \pm 0.114$ \\
& $\mathcal C_{n+1,\underline{\alpha}}^{\rm res,\cap}$
& $0.903 \pm 0.006$ & $0.952 \pm 0.004$ & $0.952 \pm 0.004$
& $4.146 \pm 0.137$ & $4.081 \pm 0.112$ \\
& $\mathcal C_{n+1,\alpha}^{\rm s\text{-}res}$
& $0.908 \pm 0.007$ & $0.954 \pm 0.004$ & $0.954 \pm 0.004$
& $4.280 \pm 0.161$ & $4.136 \pm 0.148$ \\
& $\mathcal C_{n+1,\underline{\alpha}}^{\rm s\text{-}res,\cap}$
& $0.909 \pm 0.007$ & $0.954 \pm 0.004$ & $0.954 \pm 0.004$
& $4.342 \pm 0.193$ & $4.147 \pm 0.146$ \\
& $\mathcal C_{n+1,\alpha}^{\rm q}$
& $0.907 \pm 0.006$ & $0.954 \pm 0.004$ & $0.954 \pm 0.004$
& $4.216 \pm 0.139$ & $4.133 \pm 0.147$ \\

& $\mathcal C_{n+1,\underline{\alpha}}^{\rm q,\cap}$
& $0.916 \pm 0.006$ & $0.958 \pm 0.004$ & $0.958 \pm 0.004$
& $4.379 \pm 0.148$ & $4.282 \pm 0.127$ \\

& $\mathcal C_{n+1,\underline{\alpha}}^{\rm q\text{-}sgn,\cap}$
& $0.908 \pm 0.006$ & $0.954 \pm 0.004$ & $0.954 \pm 0.004$
& $4.248 \pm 0.137$ & $4.145 \pm 0.146$ \\[4pt]
%=========================================================
% Skew-t IID
%=========================================================
Skew-$t$ IID
& Benchmark
& $0.937 \pm 0.006$ & $\bm{0.937 \pm 0.006}$ & $\bm{1.000 \pm 0.000}$
& $3.021 \pm 0.173$ & $3.023 \pm 0.153$ \\
& $\mathcal C_{n+1,\alpha}^{\rm res}$
& $0.905 \pm 0.005$ & $\bm{0.905 \pm 0.005}$ & $\bm{1.000 \pm 0.000}$
& $2.286 \pm 0.111$ & $2.218 \pm 0.091$ \\
& $\mathcal C_{n+1,\underline{\alpha}}^{\rm res,\cap}$
& $0.911 \pm 0.018$ & $0.952 \pm 0.004$ & $0.959 \pm 0.016$
& $2.633 \pm 0.135$ & $2.570 \pm 0.093$ \\
& $\mathcal C_{n+1,\alpha}^{\rm s\text{-}res}$
& $0.910 \pm 0.007$ & $\bm{0.910 \pm 0.007}$ & $\bm{1.000 \pm 0.001}$
& $2.457 \pm 0.153$ & $2.285 \pm 0.152$ \\
& $\mathcal C_{n+1,\underline{\alpha}}^{\rm s\text{-}res,\cap}$
& $0.920 \pm 0.020$ & $0.955 \pm 0.005$ & $0.966 \pm 0.017$
& $2.822 \pm 0.182$ & $2.627 \pm 0.134$ \\
& $\mathcal C_{n+1,\alpha}^{\rm q}$
& $0.911 \pm 0.010$ & $\bm{0.911 \pm 0.008}$ & $\bm{1.000 \pm 0.003}$
& $2.432 \pm 0.155$ & $2.319 \pm 0.178$ \\
& $\mathcal C_{n+1,\underline{\alpha}}^{\rm q,\cap}$
& $0.954 \pm 0.004$ & $0.954 \pm 0.004$ & $\bm{1.000 \pm 0.000}$
& $3.356 \pm 0.164$ & $3.282 \pm 0.141$ \\
& $\mathcal C_{n+1,\underline{\alpha}}^{\rm q\text{-}sgn,\cap}$
& $0.923 \pm 0.027$ & $0.954 \pm 0.004$ & $0.969 \pm 0.025$
& $2.740 \pm 0.134$ & $2.634 \pm 0.144$ \\

\hline
\end{tabular}
\end{table}

%=========================================================
% TABLE 2 — AR(1) - ACI
%=========================================================
\begin{table}[htbp]
\centering
\caption{
Empirical coverage and width of compared methods (non-exchangeable setting), reported as mean $\pm$ standard deviation over $R = 500$ MC runs. All CP intervals are constructed using the ACI framework. Global miscoverage is set at $\alpha = 0.1$, with tail-specific levels at $\alpha/2$. Bold font flags global undercoverage, tail undercoverage, or extreme tail overcoverage.
}
\label{tab:cp_aci_ar1}
\scriptsize
\setlength{\tabcolsep}{2.7pt}
\renewcommand{\arraystretch}{1.1}

\begin{tabular}{llcccccc}
\hline
\textbf{Scenario} & \textbf{Method}
& $\widehat{\mathrm{Cov}}$
& $\widehat{\mathrm{Cov}}_L$
& $\widehat{\mathrm{Cov}}_U$
& Mean width
& Median width \\
\hline

%=========================================================
% Normal AR(1)
%=========================================================
Normal AR(1)
& Benchmark
& $\bm{0.898 \pm 0.006}$ & $\bm{0.949 \pm 0.004}$ & $\bm{0.949 \pm 0.004}$
& $3.278 \pm 0.060$ & $3.283 \pm 0.052$ \\
& $\mathcal C_{n+1,\alpha}^{\rm res}$
& $0.901 \pm 0.001$ & $0.950 \pm 0.003$ & $0.950 \pm 0.003$
& $3.325 \pm 0.055$ & $3.296 \pm 0.056$ \\
& $\mathcal C_{n+1,\underline{\alpha}}^{\rm res,\cap}$
& $0.901 \pm 0.001$ & $0.951 \pm 0.001$ & $0.951 \pm 0.001$
& $3.351 \pm 0.059$ & $3.312 \pm 0.059$ \\
& $\mathcal C_{n+1,\alpha}^{\rm s\text{-}res}$
& $0.901 \pm 0.001$ & $0.950 \pm 0.003$ & $0.950 \pm 0.003$
& $3.385 \pm 0.088$ & $3.290 \pm 0.057$ \\
& $\mathcal C_{n+1,\underline{\alpha}}^{\rm s\text{-}res,\cap}$
& $0.901 \pm 0.001$ & $0.951 \pm 0.001$ & $0.951 \pm 0.001$
& $3.452 \pm 0.152$ & $3.307 \pm 0.060$ \\
& $\mathcal C_{n+1,\alpha}^{\rm q}$
& $0.901 \pm 0.001$ & $0.950 \pm 0.003$ & $0.950 \pm 0.003$
& $3.336 \pm 0.057$ & $3.293 \pm 0.057$ \\
& $\mathcal C_{n+1,\underline{\alpha}}^{\rm q,\cap}$
& $0.906 \pm 0.004$ & $0.953 \pm 0.003$ & $0.953 \pm 0.003$
& $3.390 \pm 0.061$ & $3.313 \pm 0.053$ \\
& $\mathcal C_{n+1,\underline{\alpha}}^{\rm q\text{-}sgn,\cap}$
& $0.901 \pm 0.001$ & $0.951 \pm 0.001$ & $0.951 \pm 0.001$
& $3.364 \pm 0.059$ & $3.310 \pm 0.059$ \\[4pt]
%=========================================================
% Student AR(1)
%=========================================================
Student-$t$ AR(1)
& Benchmark
& $0.909 \pm 0.007$ & $0.955 \pm 0.005$ & $0.955 \pm 0.004$
& $4.224 \pm 0.148$ & $4.231 \pm 0.133$ \\
& $\mathcal C_{n+1,\alpha}^{\rm res}$
& $0.901 \pm 0.001$ & $0.950 \pm 0.003$ & $0.950 \pm 0.003$
& $4.104 \pm 0.093$ & $4.053 \pm 0.099$ \\
& $\mathcal C_{n+1,\underline{\alpha}}^{\rm res,\cap}$
& $0.901 \pm 0.001$ & $0.951 \pm 0.001$ & $0.951 \pm 0.001$
& $4.167 \pm 0.116$ & $4.089 \pm 0.099$ \\
& $\mathcal C_{n+1,\alpha}^{\rm s\text{-}res}$
& $0.901 \pm 0.001$ & $0.950 \pm 0.003$ & $0.950 \pm 0.003$
& $4.186 \pm 0.126$ & $4.044 \pm 0.100$ \\
& $\mathcal C_{n+1,\underline{\alpha}}^{\rm s\text{-}res,\cap}$
& $0.901 \pm 0.001$ & $0.951 \pm 0.001$ & $0.951 \pm 0.001$
& $4.286 \pm 0.184$ & $4.079 \pm 0.102$ \\
& $\mathcal C_{n+1,\alpha}^{\rm q}$
& $0.901 \pm 0.001$ & $0.950 \pm 0.003$ & $0.950 \pm 0.003$
& $4.127 \pm 0.096$ & $4.051 \pm 0.100$ \\
& $\mathcal C_{n+1,\underline{\alpha}}^{\rm q,\cap}$
& $0.914 \pm 0.006$ & $0.957 \pm 0.004$ & $0.957 \pm 0.004$
& $4.351 \pm 0.165$ & $4.252 \pm 0.129$ \\
& $\mathcal C_{n+1,\underline{\alpha}}^{\rm q\text{-}sgn,\cap}$
& $0.901 \pm 0.001$ & $0.951 \pm 0.001$ & $0.951 \pm 0.001$
& $4.186 \pm 0.117$ & $4.084 \pm 0.102$ \\[4pt]
%=========================================================
% Skew-t AR(1) ACI
%=========================================================
Skew-$t$ AR(1)
& Benchmark
& $0.937 \pm 0.006$ & $\bm{0.937 \pm 0.006}$ & $\bm{1.000 \pm 0.000}$
& $3.020 \pm 0.172$ & $3.023 \pm 0.153$ \\
& $\mathcal C_{n+1,\alpha}^{\rm res}$
& $0.901 \pm 0.001$ & $\bm{0.901 \pm 0.001}$ & $\bm{1.000 \pm 0.001}$
& $2.233 \pm 0.086$ & $2.155 \pm 0.100$ \\
& $\mathcal C_{n+1,\underline{\alpha}}^{\rm res,\cap}$
& $0.902 \pm 0.002$ & $0.951 \pm 0.001$ & $0.951 \pm 0.001$
& $2.634 \pm 0.107$ & $2.559 \pm 0.088$ \\
& $\mathcal C_{n+1,\alpha}^{\rm s\text{-}res}$
& $0.901 \pm 0.001$ & $\bm{0.902 \pm 0.002}$ & $\bm{0.999 \pm 0.001}$
& $2.310 \pm 0.114$ & $2.147 \pm 0.099$ \\
& $\mathcal C_{n+1,\underline{\alpha}}^{\rm s\text{-}res,\cap}$
& $0.902 \pm 0.002$ & $0.951 \pm 0.001$ & $0.951 \pm 0.002$
& $2.747 \pm 0.167$ & $2.554 \pm 0.090$ \\
& $\mathcal C_{n+1,\alpha}^{\rm q}$
& $0.901 \pm 0.002$ & $\bm{0.902 \pm 0.002}$ & $\bm{0.999 \pm 0.002}$
& $2.272 \pm 0.091$ & $2.163 \pm 0.099$ \\
& $\mathcal C_{n+1,\underline{\alpha}}^{\rm q,\cap}$
& $0.951 \pm 0.001$ & $0.951 \pm 0.001$ & $\bm{1.000 \pm 0.000}$
& $3.311 \pm 0.150$ & $3.219 \pm 0.126$\\
& $\mathcal C_{n+1,\underline{\alpha}}^{\rm q\text{-}sgn,\cap}$
& $0.903 \pm 0.003$ & $0.951 \pm 0.001$ & $0.952 \pm 0.003$
& $2.669 \pm 0.123$ & $2.564 \pm 0.091$ \\
\hline
\end{tabular}
\end{table}

\section{An Application to Financial Forecasting} \label{sec: app}

As motivated in {\bf Example 1}, financial return forecasting presents a canonical problem of two-sided asymmetry: an investor simultaneously cares about downside risk (extreme losses) and upside potential (large gains), yet the left tail carries far more severe consequences and demands stricter probabilistic control. 
%By specifying separate miscoverage rates for each tail, the proposed framework yields a tail-calibrated interval with pre-specified control guarantees on both extremes of the return distribution.
Value-at-Risk (VaR) formalizes the left-tail objective, quantifying the maximum loss that an investment  can experience with a specified probability over a specific period of time~\citep{mcneil2015quantitative}. Its accurate estimation is critical for regulatory compliance, capital allocation, and financial risk management. Despite being central, a framework focused exclusively on VaR foregoes characterization of the upper tail, which carries its own financial relevance, giving explicit information on potential gains and guiding portfolio optimization~\citep{chen2008two, farinelli2008beyond}.

Formally, let $\{Y_n\}_{n \geq 1}$ denote the sequence of daily log-returns of a financial instrument, where $Y_n = \log\!\left(\frac{P_n}{P_{n-1}}\right)$, with $P_n$ and $P_{n-1}$ being the price at times $n$ and $n-1$, respectively. Define $\mathcal{F}_{n-1} = \{Y_1,\dots,Y_{n-1}\}$ as the information set available at time $t-1$. For return series, the VaR at confidence level $\alpha^- \in (0,1)$
is defined as the lower conditional quantile of the
return distribution, namely
\begin{equation*}
\mathrm{VaR}_{\alpha^-}(Y_n)
:=
\inf\bigl\{x \in \mathbb{R} :
\mathbb{P}(Y_n \le x \mid \mathcal{F}_{n-1}) \ge \alpha^-
\bigr\},\quad n \geq 1.
\end{equation*}

%Equivalently, when the conditional return distribution is continuous, we have $\mathbb{P}(y_t \le \mathrm{VaR}_{t,\alpha}\mid \mathcal{F}_{t-1})=\alpha$.
Typical approaches for VaR estimation rely on parametric/semiparametric assumptions or use historical and Monte Carlo simulations. These approaches often rely on strong distributional assumptions and provide no finite-sample guarantees. Even asymptotic validity may fail under model misspecification, nonstationarity,
or time-varying tail behavior.

Conformal methods are especially well suited to this task for two key reasons. First, their distribution-free nature allows valid risk guarantees to be obtained
without strong assumptions on return dynamics. Second, adaptive
conformal procedures (e.g., based on ACI) naturally accommodate non-stationarity, enabling the
resulting VaR estimates to adapt and remain reliable in evolving market conditions. \cite{fantazzini2023adaptive} explored this direction by constructing a classical \emph{two-sided} CQR prediction interval for the returns and interpreting its lower bound as a VaR estimate. While this approach guarantees global marginal coverage, it provides no explicit guarantee for the left tail in isolation, and, symmetrically, none for the right tail. The framework proposed here closes this gap.

In this work, once a tail-calibrated two-sided CP interval for the next-period log-return $Y_{n+1}$ has been constructed following Eq.~\eqref{eq:intersection_interval}, we define the conformal VaR estimator as 
\begin{align*}
    \widehat{\mathrm{VaR}}_{\alpha^-}(Y_{n+1}) = \widehat{\mathrm{VaR}}_{n+1,\alpha^-} := L_{n+1}.
\end{align*}
In virtue of the tail-specific validity provided in Section~\ref{sec: theory}, this construction guarantees $\mathbb{P}\!\left(Y_{n+1} < \widehat{\mathrm{VaR}}_{n+1,\alpha^-}\right) \leq \alpha^-$, resulting in a valid distribution-free approach for VaR estimation. Notably, traditional VaR backtests, including Kupiec or Christoffersen tests~\citep{kupiec1995techniques,christoffersen1998evaluation}, assess
{\it ex post} whether a VaR model attains the desired violation or miscoverage frequency. In contrast, with the proposed CP framework, this property is achieved theoretically {\it ex ante} by design. An application to a set of financial stocks will now enforce the theoretical validity of the proposed tail-calibrated CP approach in comparison to classical two-sided CP intervals. 
%through an online feedback mechanism, effectively embedding backtesting into the construction of the VaR itself.

% \paragraph*{Benchmark Volatility Model for Daily Returns}
\textbf{Prediction model.} As a benchmark model specification for the one-step-ahead VaR forecasting, we consider a GARCH(1,1) model with constant conditional mean $\mu$ and time-varying conditional volatility $\sigma_n$, defined as
\begin{align*}
Y_n &= \mu + \varepsilon_n, \qquad
\varepsilon_n = \sigma_n z_n, \qquad
z_n \sim t_\nu,\qquad n = 1,2,\dots \\
\sigma_n^2 &= \omega + \alpha \varepsilon_{n-1}^2 + \beta \sigma_{n-1}^2,
\end{align*}
where $\{z_n\}_{n \geq 1}$ is an i.i.d. sequence of standardized Student's $t_\nu$ random variables with $\nu$ degrees of freedom. This model specification captures salient features of financial return series, including volatility clustering, conditional heteroskedasticity, and heavy tails, through the interaction of dynamic conditional variance and non-Gaussian innovations. 

Beyond being employed as the underlying predictive model in all CP procedures, this model also serves as a standard parametric benchmark. %In the latter case, a two-sided conformal prediction interval for returns is first constructed, and its lower endpoint is subsequently interpreted as the VaR estimate.
The one-step-ahead VaR at confidence level $\alpha^-$ can be obtained as
\begin{equation*}
\mathrm{VaR}_{n+1,\alpha^-}
=
\hat{\mu}_{n+1}
+
t^{-1}_{\alpha^-,\hat{\nu}}
\sqrt{\frac{\hat{\nu}-2}{\hat{\nu}}}\,
\hat{\sigma}_{n+1},
\end{equation*}
where $\hat{\mu}_{n+1}$ and $\hat{\sigma}_{n+1}^2$ denote the one-step-ahead forecasts of the conditional mean and variance, respectively, while $t^{-1}_{\alpha^-,\hat{\nu}}$ is the $\alpha^-$-quantile of the Student's $t$ distribution with estimated degrees of freedom $\hat{\nu}$.

\textbf{Financial data.} We consider three exchange-traded funds (ETFs) with different risk characteristics: (a) SPY, which tracks the S\&P 500 and represents the broad U.S. equity market; (b) TQQQ, a leveraged ETF linked to the Nasdaq-100 and characterized by high volatility; and (c) XLE, which tracks the U.S. energy sector and is strongly exposed to commodity-price shocks. TQQQ and XLE are analyzed over the 2019-2021 period, selected to assess the COVID-19 market shock, with abrupt volatility changes, severe downside movements, and negatively skewed return distributions. SPY is evaluated over a longer time horizon, 2018-2025, to evaluate multiple market regimes, including turbulent and more stable periods.

\textbf{CP specification.} Since financial return time series are generally nonexchangeable, the more robust DtACI framework is employed, using all three conformal scores of interest. The following grid of learning-rate parameters for DtACI is considered: $\gamma \in \{0.005,\;0.008,\;0.010,\;0.015,\;0.020\}$. For the proposed CP approach, the target lower-tail miscoverage level is set to $\alpha^- = \alpha/2 = 0.10$; for the classical two-sided CP intervals, %corresponding to a nominal lower-tail coverage of $90\%$ for the one-sided conformal procedures. For the two-sided procedures, we set 
the overall miscoverage level is set to $\alpha = 0.20$. %, with the total error probability allocated symmetrically across the two tails, i.e., $\alpha^{+}=\alpha^{-}=0.10$.
This ensures that the lower-tail coverage remains directly comparable across one-sided and two-sided methods. 

\textbf{Results.} Figure \ref{fig:coverage_prediction_intervals} summarizes the empirical performance of the compared methods across the considered assets. The left panels report the average empirical lower-tail coverage over time, reflecting the risk-management performance using VaR.
Across all assets, both the GARCH-based VaR forecasts and the classical two-sided CP procedures systematically exhibit undercoverage relative to the nominal target $\alpha^-$. This is accompanied by simultaneous overcoverage with respect to the nominal upper-tail target level $\alpha^+$ (see Figure 1 in the Supplementary material E). By contrast, the one-sided procedures consistently achieve the closest alignment with the desired coverage level, confirming their theoretical validity.

The right panels of Figure~\ref{fig:coverage_prediction_intervals} compare the resulting log-return forecast intervals, along with the observed time series. %All methods, except those based on classical residual scores $s^{\text{res}}$, react sharply during periods of market stress, most notably during the 2020 COVID-19 crisis, when VaR estimates increase substantially.  %Compared to GARCH forecasts, all CP procedures adapt more rapidly to changing volatility conditions. 
The one-sided CP estimates are slightly more conservative in stressed periods, which is desirable from a downside risk management perspective. % When comparing the different scores, we observe that the standardized residual and the quantile scores, as well as the benchmark, tend to react more sharply to changes in market conditions than the residual-based version.
When comparing the different methods, we observe that the standardized residual $s^{\text{s-res}}$ and the signed quantile residual $s^{\text{q-sgn}}$, as well as the GARCH model, react more sharply to changes in market conditions (especially during the 2020 COVID-19 crisis) than the residual $s^{\text{res}}$. This behavior is expected in heteroskedastic environments such as financial returns, since these methods explicitly account for the conditional variance.
%Overall, the empirical evidence supports the use of one-sided adaptive conformal prediction for VaR estimation. By directly controlling lower-tail exceedances in an online manner, these methods provide more reliable risk forecasts than standard parametric approaches while preserving flexibility under nonstationary market conditions.

%=========================================================
% MULTIPANEL — LOWER COVERAGE + PI
%=========================================================
\begin{figure}[htbp]
\centering

\includegraphics[
width=\textwidth
]{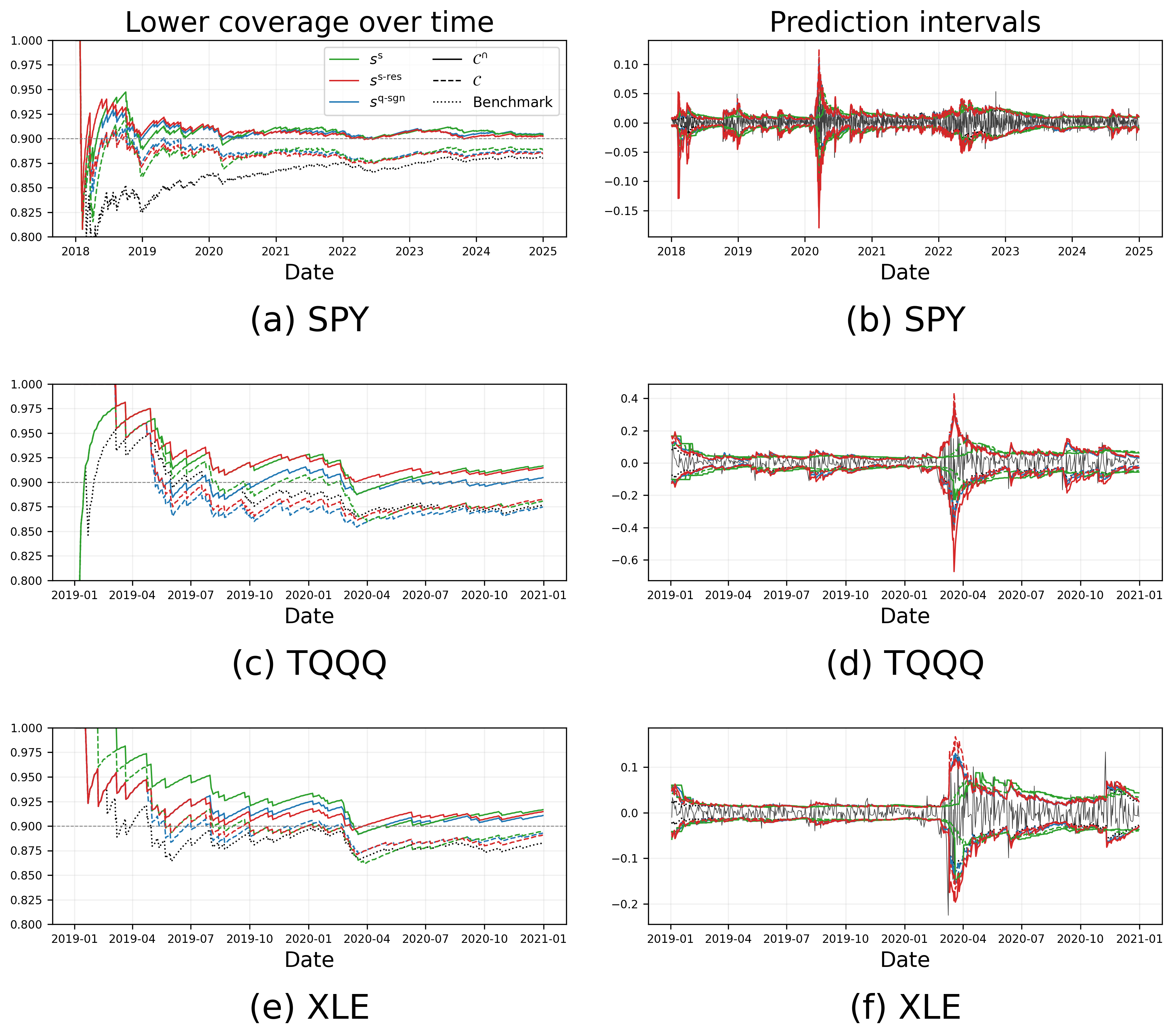}

\caption{
Comparison of lower-tail coverage (left panels) and prediction intervals (right panels) over time for SPY, TQQQ, and XLE obtained using the benchmark forecast model GARCH-$t$ in isolation and in combination with the CP procedures under the residual $s^{\text{res}}$ (res), s-residual $s^{\text{s-res}}$ (s-res), and quantile scores $s^{\text{q-sgn}}$ (q-sgn).}
\label{fig:coverage_prediction_intervals}

\end{figure}

\section{Conclusion}
This paper develops a conformal prediction methodology designed to control lower- and upper-tail prediction errors separately, rather than only enforcing the usual global marginal coverage requirement at level $1-\alpha$. %Within the split CP framework, we construct separate lower and upper one-sided prediction intervals and combine them through intersection to obtain a two-sided prediction set. 
We establish coverage guarantees for each tail and derive the corresponding global coverage properties of the resulting interval. The analysis covers both the exchangeable case, where the guarantees are finite-sample, and the non-exchangeable case, where adaptive procedures yield asymptotic validity, as well as short-term regret bounds. The simulation results indicate that the proposed construction substantially improves tail-specific calibration compared with standard two-sided conformal intervals, particularly in skewed and heavy-tailed environments. A financial application further confirms these findings: only the intersection-based procedures are able to deliver reliable tail-specific coverage, at the cost of slightly less efficient prediction intervals.

Several directions for future research emerge naturally from the present work. First, while the proposed intersection-based construction provides explicit tail-specific coverage guarantees, it does not recover the tight finite-sample global validity achieved by classical two-sided conformal prediction. %Developing joint calibration mechanisms capable of simultaneously preserving directional validity and optimal global coverage efficiency therefore represents an important open problem. Such 
Extensions may consider coupled-tail calibration procedures beyond the independent treatment of the two tails adopted here. Second, extending the framework toward weighted conformal inference schemes \citep{tibshirani2019conformal,schmitt2024taming} appears particularly promising in non-stationary settings, where adaptive weighting mechanisms may allow the calibration procedure to respond more effectively to evolving distributional regimes and covariate shifts over time. Third, integrating the proposed tail-specific construction within a Bayesian conformal prediction framework~\citep{deliu2026interplay} represents another promising direction. By incorporating prior and posterior uncertainty into the calibration mechanism, such an extension may lead to more informative prediction intervals and potential efficiency gains.

\bibliography{main}

\newpage
\appendix
\begin{center}  
{\bf \Large Supplementary Material to\\
\textbf{\textit{ Conformal Prediction Intervals with Tail-Specific Guarantees}}}
\end{center}

\section{Technical background on CQR}

In quantile regression, given a feature $X_{n+1}=x$, one can construct $1-\alpha$ prediction intervals for $Y_{n+1}$, by considering the conditional distribution $F(y \mid X_{n+1} = x)$, with $F$ denoting the cumulative density function, and modeling its \textit{conditional quantiles}~\citep{koenker2005quantile}. For a given level $\tau \in (0,1)$, the {\it conditional quantile function} at $X_{n+1}=x$ is defined as
\begin{equation*}
q_\tau(x) = \inf\{ y \in \mathbb{R} : F(y \mid X_{n+1} = x) \ge \tau \}.
\label{eq:quantile_function}
\end{equation*}

A quantile regression estimator, say $\widehat{q}_{\tau}(x)$, is obtained as the solution of the problem
\begin{equation}
\widehat{q}_\tau(x)=f(x;\widehat{\theta}_\tau),
\qquad
\widehat{\theta}_\tau
=
\arg\min_{\theta}
\frac{1}{n}\sum_{i=1}^{n}
\ell_\tau\!\bigl(Y_i, f(X_i;\theta)\bigr)
+
\mathcal{R}(\theta),
\label{eq:quantile_regression}
\end{equation}
where $f(x;\theta)$ is a parametric or nonparametric model (e.g., linear, tree-based, or neural net) posited for the quantile function, $\mathcal{R}(\theta)\geq 0$ denotes a regularization penalty (set to zero in the unpenalized case), and $\ell_\tau(y,q)$ is the ``pinball'' loss function defined as
\begin{equation*}
\ell_\tau(y,q)
:=
\tau (y-q)-\min\{0,y-q\}
=
\begin{cases}
(1-\tau)(q-y), & y<q,\\[2pt]
\tau (y-q), & y\ge q,
\end{cases}
\label{eq:pinball_loss_qr}
\end{equation*}

On this basis, a $1-\alpha$ conditional prediction interval can be obtained as
\begin{equation*}
\mathcal{}{C}(x) = \big[\widehat{q}_{\alpha/2}(x),\ \widehat{q}_{1-\alpha/2}(x)\big],
\label{eq:conditional_interval}
\end{equation*}
which, nonetheless, since the estimated quantiles $\widehat{q}_\tau(x)$ are subject to model misspecification and finite-sample errors, is not ensured to have nominal $1-\alpha$ coverage.

CP provides a correction mechanism to quantile regression to achieve valid coverage without assumptions on the distribution of $(X,Y)$, except exchangeability. Following the split CP rationale, first, a lower and an upper quantile regression model is fitted on $\mathcal{D}_{n_{t}}$ to get:
\[
\widehat{q}_{\alpha/2, n_{t}}(x) = f(x; \widehat{\theta}_{\alpha/2}),
\quad
\widehat{q}_{1-\alpha/2, n_{t}}(x) = f(x; \widehat{\theta}_{1-\alpha/2}),
\]
each minimizing its respective optimization problem as in Eq.~\eqref{eq:quantile_regression}. Then, the calibration set is used to compute quantile conformity scores
\begin{equation*}
s^{\text{q}}_i = \max\big\{
\widehat{q}_{\alpha/2, n_{t}}(X_i) - Y_i,\;
Y_i - \widehat{q}_{1-\alpha/2, n_{t}}(X_i)
\big\},\qquad i = 1,\dots, n_{c},
\label{eq:cqr_scores}
\end{equation*}
and the CQR interval for $X_{n+1}$ is obtained as
\begin{equation*}
\mathcal{C}^q_{n+1, \alpha}=
\big[
\widehat{q}_{\alpha/2, n_{t}}(X_{n+1}) - Q^{\text{q}}_{\alpha, n_{c}},\;
\widehat{q}_{1-\alpha/2, n_{t}}(X_{n+1}) + Q^{\text{q}}_{\alpha, n_{c}}
\big],
\label{eq:cqr_interval}
\end{equation*}
with $Q^{\text{q}}_{\alpha, n_{c}}$ the empirical $(1-\alpha)(1 + 1/n_{c})$-quantile of the calibration scores $S^{\text{q}}_1,\dots, S^{\text{q}}_{n_{c}}$, as per standard split CP.

\section{Implementation Details and Hyperparameter Choices} \label{sec:implementation}

This section collects implementation details for the DtACI procedure that
are needed for reproducibility but would interrupt the flow of the main text.
We describe the hyperparameter choices adopted throughout the simulations and
the empirical application. We also provide the pseudocode for the tail-calibrated DtACI in Algorithm \ref{alg:onesidedtaci}.

\paragraph*{Weight regularization.}
To prevent weight degeneracy and ensure persistent exploration across all
candidates, the updated weights are mixed with the uniform distribution as
\[
w_{i+1}^j
=
(1-\sigma)\,\tilde{w}_{i+1}^j
+
\sigma\,\frac{1}{k}\sum_{l=1}^k \tilde{w}_{i+1}^l,\qquad i\ge n+1,
\]
where $\sigma\in(0,1)$ controls the amount of exploration.

\paragraph*{Tuning parameters.}
Decaying choices of $\eta_i^\pm$ and $\sigma_i^\pm$ are appropriate only
when the degree of nonstationarity is itself stable over time, an assumption
that is often unrealistic. We therefore use constant tuning parameters,
following \citet{gibbs2024conformal}, namely
\[
\eta^\pm
=
\sqrt{\frac{3}{|I|}}
\sqrt{
\frac{\log(k|I|)+2}
{(1-\alpha^\pm)^2(\alpha^\pm)^2}
},
\qquad
\sigma^\pm=\frac{1}{2|I|},
\]
with $|I|=500$. Although constant or slowly varying choices may induce an
asymptotic bias in the average miscoverage rate, this effect has been found
negligible in practice \citep{gibbs2024conformal}.

\paragraph*{Deterministic aggregation.}
To avoid the additional randomness induced by sampling $\alpha_i^\pm$, we
use the deterministic aggregated level $\bar{\alpha}_i^\pm := \sum_{l=1}^k
p_i^{l,\pm}\alpha_i^{l,\pm}$ in place of drawing from $\sum_{l=1}^k
p_i^{l,\pm}\delta_{\alpha_i^{l,\pm}}$. This variant satisfies the same
regret guarantee and behaves almost identically in practice
\citep{gibbs2024conformal}. Accordingly, we adopt $\bar{\alpha}_i^\pm$
throughout both the simulations and the empirical application.

\begin{algorithm}[t]
\caption{DtACI intervals with tail-specific guarantees}
\label{alg:onesidedtaci}
\begin{algorithmic}[1]
\REQUIRE Data $\mathcal D_n$; tail miscoverage levels $\alpha^-,\alpha^+$; candidate learning rates $\{\gamma_j^\pm\}_{j=1}^k$; initial adaptive levels $\{\alpha_{n+1}^{j,\pm}\}_{j=1}^k$; exploration parameters $\sigma^\pm$; exponential weighting rates $\eta^\pm$.
\STATE Initialize expert weights
$
w_{n+1}^{j,\pm} \leftarrow 1,
\quad \forall j=1,\dots,k.
$

\FOR{$i = n+1,n+2,\ldots,N$}

    \STATE Define aggregation probabilities $p_i^{j,\pm}
    :=
    \frac{w_i^{j,\pm}}
    {\sum_{l=1}^{k} w_i^{l,\pm}},
    \qquad j=1,\dots,k.$
    
    \STATE Set
    $
    \alpha_i^\pm
    $
    by sampling from
    $
    \{\alpha_i^{j,\pm}\}_{j=1}^k
    $
    with probabilities
    $
    \{p_i^{j,\pm}\}_{j=1}^k
    $,
    or equivalently take
    $
    \alpha_i^\pm
    =
    \sum_{j=1}^k p_i^{j,\pm}\alpha_i^{j,\pm}.
    $
    \STATE Construct intervals
    $
    \mathcal{C}_{i,\alpha_i^-}^{L}
    $
    and
    $
    \mathcal{C}_{i,\alpha_i^+}^{U}
    $
    according to Table~\ref{tab:onesided_cp_scores}.

    \STATE Construct the intersection interval $\mathcal{C}_{i,\underline{\alpha}_{i}}^{\cap}
    :=
    \mathcal{C}_{i,\alpha_i^-}^{L}
    \cap
    \mathcal{C}_{i,\alpha_i^+}^{U}$, with $\underline{\alpha}_{i}=(\alpha_i^-,\alpha_i^+)$

    \STATE Define $\beta_i^-:=
    \sup\Bigl\{
    \beta \in [0,1]:
    Y_i \in \mathcal{C}_{i,\beta}^{L}
    \Bigr\}$, $\beta_i^+
    :=
    \sup\Bigl\{
    \beta \in [0,1]:
    Y_i \in \mathcal{C}_{i,\beta}^{U}
    \Bigr\}.$

    \STATE $\tilde{w}_i^{j,\pm}
    \leftarrow
    w_i^{j,\pm}
    \exp\!\bigl(
    -\eta^\pm
    \ell(\beta_i^\pm,\alpha_i^{j,\pm})
    \bigr),
    \qquad j=1,\dots,k.$

    \STATE $\tilde{W}_i^\pm
    :=
    \sum_{j=1}^{k}
    \tilde{w}_i^{j,\pm}.$

    \STATE $w_{i+1}^{j,\pm}
    \leftarrow
    (1-\sigma^\pm)\tilde{w}_i^{j,\pm}
    +
    \sigma^\pm \tilde{W}_i^\pm/k,
    \qquad j=1,\dots,k.$

    \STATE $\mathrm{err}_i^{j,-}:=\mathbf{1}\{Y_i \notin \mathcal{C}_{i,\alpha_i^{j,-}}^{L}\}, \quad
    \mathrm{err}_i^{j,+}
    :=
    \mathbf{1}\{Y_i \notin \mathcal{C}_{i,\alpha_i^{j,+}}^{U}\},
    \qquad j=1,\dots,k.$

    \STATE $\alpha_{i+1}^{j,\pm}
    \leftarrow
    \alpha_i^{j,\pm}
    +
    \gamma_j^\pm
    \bigl(
    \alpha^\pm
    -
    \mathrm{err}_i^{j,\pm}
    \bigr),
    \qquad j=1,\dots,k.$

\ENDFOR
\end{algorithmic}
\end{algorithm}

\section{Proofs for the Results in Section~4}

For clarity and completeness, this section collects the formal statements of all theorems and propositions presented in Section~4 of the main text, together with their proofs.

\begin{proposition}[Tail-specific coverage guarantees]
\label{_thm:onesided_validity_three_families}
Under exchangeability, the lower/upper one-sided split CP intervals in Table~\ref{tab:onesided_cp_scores} satisfy tail-specific $1-\alpha^-$ / $1-\alpha^+$ coverage guarantees:
\begin{equation}
\mathbb P\!\left(Y_{n+1}\in\mathcal C_{n+1,\alpha^-}^L\right)\ge 1-\alpha^-,
\qquad
\mathbb P\!\left(Y_{n+1}\in\mathcal C_{n+1,\alpha^+}^{U}\right)\ge 1-\alpha^+.
\label{_eq:onesided_lower_bounds_general}
\end{equation}
Moreover, if the corresponding scores are almost surely distinct, then
\begin{equation}
\mathbb P\!\left(Y_{n+1}\in \mathcal C_{n+1,\alpha^-}^L\right)\le 1-\alpha^-+\frac{1}{n_{c}+1},
\quad 
\mathbb P\!\left(Y_{n+1}\in\mathcal C_{n+1,\alpha^+}^{U}\right)\le 1-\alpha^+ + \frac{1}{n_{c}+1}.
\label{_eq:onesided_upper_bounds_general}
\end{equation}
\end{proposition}
\begin{proof}
The proof follows verbatim the theorems in \cite{lei2018distribution}
\end{proof}
\begin{theorem}[Global coverage guarantees]
\label{_thm:two_sided_coverage}
Assume all conformity score functions are quasi-convex in $y$. Let $L_{n+1}$ and $U_{n+1}$ be the lower and upper one-sided CP interval bounds for $n+1$, with miscoverage levels $\alpha^{\pm}\in(0,1)$. Let $\underline{\alpha}
:=
(\alpha^-,\alpha^+)$, and define %the associated CP interval obtained as the intersection of the two one-sided CP intervals, that is,
\[
\mathcal C_{n+1,\underline{\alpha}}^{\cap}
:=
[L_{n+1},U_{n+1}]
=
\mathcal C_{n+1,\alpha^-}^{L}
\cap
\mathcal C_{n+1,\alpha^+}^{U}.
\]
Under exchangeability, if the one-sided marginal guarantees  in Eq.~\eqref{_eq:onesided_lower_bounds_general} hold, then
\begin{equation*}
1-(\alpha^{-}+\alpha^{+})
\;\le\;
\mathbb P\!\left(Y_{n+1}\in \mathcal C_{n+1,\underline{\alpha}}^{\cap}\right).
\end{equation*}
Moreover, if the corresponding scores are almost surely distinct, so that \eqref{_eq:onesided_upper_bounds_general} holds, then
\begin{equation*}
1-(\alpha^{-}+\alpha^{+})
\;\le\;
\mathbb P\!\left(Y_{n+1}\in \mathcal C_{n+1,\underline{\alpha}}^{\cap}\right)
\;\le\;
1-\max\{\alpha^{-},\alpha^{+}\}
+\frac{1}{n_{c}+1}.
\end{equation*}
In particular, if $\alpha^{-}=\alpha^{+}=\alpha/2$, then $1-\alpha
\;\le\;
\mathbb P\!\left(Y_{n+1}\in \mathcal C_{n+1,\underline{\alpha}}^{\cap}\right)
\;\le\;
1-\frac{\alpha}{2}
+\frac{1}{n_{c}+1}$.
\end{theorem}
\begin{proof}
We establish the lower (A) and upper (B) bounds separately. 

(A) We consider the complement event of $\{Y_{n+1}\in C_{n+1,\underline{\alpha}}^{\cap}\}$, that is, $\{Y_{n+1}\notin \mathcal C_{n+1,\underline{\alpha}}^{\cap}\}
=
\{Y_{n+1}<L_{n+1}\}\cup\{Y_{n+1}>U_{n+1}\}$. %That is, $Y_{n+1}$ falls outside the interval whenever it lies below the lower threshold or above the upper threshold. 
Applying Boole’s inequality we get $\mathbb P(Y_{n+1}\notin \mathcal C_{n+1,\underline{\alpha}}^{\cap})
\le
\mathbb P(Y_{n+1}<L_{n+1})
+
\mathbb P(Y_{n+1}>U_{n+1})$. By Proposition~\ref{_thm:onesided_validity_three_families}-\eqref{_eq:onesided_lower_bounds_general}, these two probabilities are bounded by
$\alpha^{-}$ and $\alpha^{+}$, respectively. Hence, $\mathbb P(Y_{n+1}\notin \mathcal C_{n+1,\underline{\alpha}}^{\cap})
\le
\alpha^{-}+\alpha^{+}$. Taking complements yields $\mathbb P(Y_{n+1}\in \mathcal C_{n+1,\underline{\alpha}}^{\cap})
\ge
1-(\alpha^{-}+\alpha^{+})$.

(B) By construction, $\{Y_{n+1}\in C_{n+1,\underline{\alpha}}^{\cap}\}
=
\{Y_{n+1}\ge L_{n+1}\}\cap\{Y_{n+1}\le U_{n+1}\}$. Since the intersection of two events is contained
in each of them, we have $\{Y_{n+1}\in C_{n+1,\underline{\alpha}}^{\cap}\}
\subseteq
\{Y_{n+1}\ge L_{n+1}\}$ and $\{Y_{n+1}\in C_{n+1,\underline{\alpha}}^{\cap}\}
\subseteq
\{Y_{n+1}\le U_{n+1}\}$. By monotonicity of probability measures, it follows that $\mathbb P(Y_{n+1}\in C_{n+1,\underline{\alpha}}^{\cap})
\le
\min\bigl\{
\mathbb P(Y_{n+1}\ge L_{n+1}),\,
\mathbb P(Y_{n+1}\le U_{n+1})
\bigr\}$. %Substituting the finite-sample one-sided upper bounds yields
By Proposition~\ref{_thm:onesided_validity_three_families}-\eqref{_eq:onesided_upper_bounds_general}, we get
\begin{align*}
\mathbb P\!\left(Y_{n+1}\in C_{n+1,\underline{\alpha}}^{\cap}\right)
&\le
\min\!\Big\{
1-\alpha^{-}+\frac{1}{n_{c}+1},\,
1-\alpha^{+}+\frac{1}{n_{c}+1}
\Big\} \\
&=
1-\max\{\alpha^{-},\alpha^{+}\}
+\frac{1}{n_{c}+1}.
\end{align*}
This concludes the proof.
\end{proof}

\begin{lemma}
\label{lem:aci_lower}
Under the same conditions as \citet{gibbs2021adaptive}, namely, that the empirical quantile functions $\widehat{Q}_{\alpha,n_{c}}^{L,U}$ are non-decreasing with $\widehat{Q}_{\alpha,n_{c}}^{L,U}(x)=-\infty$ for $x<0$ and $\widehat{Q}_{\alpha,n_{c}}^{L,U}(x)=+\infty$ for $x>1$, with probability one, for all $i \in \mathbb N$,
\[
\alpha_i^{\pm} \in [-\gamma,\,1+\gamma].
\]
\end{lemma}

\begin{proof}
The proof follows verbatim the argument of Lemma 4.1 in
\citep{gibbs2021adaptive}.
\end{proof}

\begin{proposition}[Tail-specific coverage guarantees under ACI]
\label{_prop:aci_lower}
Under the same conditions as \citet{gibbs2021adaptive}, namely, that the empirical quantile functions $\widehat{Q}_{\alpha,n_{c}}^{L,U}$ are non-decreasing with $\widehat{Q}_{\alpha,n_{c}}^{L,U}(x)=-\infty$ for $x<0$ and $\widehat{Q}_{\alpha,n_{c}}^{L,U}(x)=+\infty$ for $x>1$, for a given $\alpha_1^\pm\in[0,1]$, with probability one, for all $N \in \mathbb N$,
\[
\left|
\frac{1}{N}\sum_{i=1}^N \mathrm{err}_i^{\pm}
-
\alpha^{\pm}
\right|
\le
\frac{\max\{\alpha_1^{\pm},\,1-\alpha_1^{\pm}\} + \gamma}{N\gamma}.
\]
In particular, $\frac{1}{N}\sum_{i=1}^N \mathrm{err}_i^{\pm}
\;\xrightarrow[N\to\infty]{\mathrm{a.s.}}\;
\alpha^{\pm}$.
\end{proposition}
\begin{proof}
The proof follows verbatim the argument of Proposition 4.1 in
\citep{gibbs2021adaptive}.  For completeness, it is reported below.\\
We present the proof for the lower-tail case; the upper-tail case follows analogously. We start from the adaptive update defining the lower-tail ACI recursion:
\[
\alpha_{i+1}^-
=
\alpha_i^- + \gamma\bigl(\alpha^- - \mathrm{err}_i^-\bigr),
\qquad i\in\mathbb N.
\]

Summing both sides from $i=1$ to $N$ yields
\[
\sum_{t=1}^N \bigl(\alpha_{t+1}^- - \alpha_t^-\bigr)
=
\sum_{t=1}^N \gamma\bigl(\alpha^- - \mathrm{err}_t^-\bigr).
\]
The left-hand side telescopes, giving
\[
\alpha_{N+1}^- - \alpha_1^-
=
\gamma \sum_{t=1}^N
\bigl(\alpha^- - \mathrm{err}_t^-\bigr).
\]

Rearranging terms, we obtain
\[
\frac{1}{T}\sum_{t=1}^N \mathrm{err}_t^-
-
\alpha^-
=
\frac{\alpha_1^- - \alpha_{T+1}^-}{N\gamma}.
\]

We now bound the numerator on the right-hand side.
By Lemma~\ref{lem:aci_lower} we have with probability one that
\[
\alpha_{N+1}^- \in [-\gamma,\,1+\gamma].
\]
 
For any $x \in [a,\,b]$ and $S\in\mathbb{R}$, we have the elementary inequality
\[
|S - x|
\le
\max\bigl\{ S - a,\ b - S \bigr\}.
\]
Applying this inequality with $x = \alpha_{N+1}^{-},  a=-\gamma, b=1+\gamma$ and $S=\alpha_1^{-}\in [0,1]$ yields

\[
\bigl|\alpha_1^- - \alpha_{N+1}^-\bigr|
\le
\max\bigl\{ \alpha_1^- + \gamma,\ (1+\gamma) - \alpha_1^- \bigr\}
=
\max\bigl\{ \alpha_1^-,\,1-\alpha_1^- \bigr\} + \gamma.
\]

Dividing both sides of previous equation by $N\gamma$ therefore gives
\[
\left|
\frac{1}{N}\sum_{t=1}^N \mathrm{err}_t^-
-
\alpha^-
\right|
\le
\frac{\max\{\alpha_1^-,\,1-\alpha_1^-\} + \gamma}{N\gamma},
\]
with probability one. This establishes the stated finite-sample bound.
Since the right-hand side converges to zero as $N\to\infty$, the almost sure
convergence follows immediately.
\end{proof}

\begin{theorem}[Global coverage guarantees under ACI]
\label{_thr:aci_total}
Under the same conditions as Proposition~\ref{_prop:aci_lower}, let
$\mathrm{err}_i^{\mathrm{tot}} := \mathrm{err}_i^- + \mathrm{err}_i^+$
and $\alpha^{\mathrm{tot}} := \alpha^- + \alpha^+$ denote the observed and target miscoverage levels of the intersection interval $\mathcal{C}_{i,\underline{\alpha}_i}^{\cap}$, obtained from two independent ACI procedures run with the 
%Consider two one-sided ACI procedures run in parallel, one for the lower tail and one for the upper tail. Define $ \mathrm{err}_i^- := \mathbf 1\{Y_i \notin \mathcal C_{i,\alpha_i^-}^{L}\}$ and $ \mathrm{err}_i^+ := \mathbf 1\{Y_i \notin \mathcal C_{i,\alpha_i^+}^{U}\}$, where $\alpha_i^-$ and $\alpha_i^+$ are the adaptive lower- and upper-tail miscoverage levels, respectively, with target levels $\alpha^-$ and $\alpha^+$.
%Assume that both ACI procedures use the 
same step size $\gamma>0$, so that $\alpha_i^\pm \in [-\gamma, 1+\gamma]$ almost surely for all $i$ (Lemma~B.1). Then, with probability one, for every $N\in\mathbb N$,
\[
\left|
\frac{1}{N}\sum_{i=1}^N \mathrm{err}_i^{\mathrm{tot}}
-
\alpha^{\mathrm{tot}}
\right|
\le
\frac{
\max\{\alpha_1^-+\alpha_1^+,\;2-(\alpha_1^-+\alpha_1^+)\}+2\gamma
}{
N\gamma
}.
\]
where $\mathrm{err}_i^{\mathrm{tot}}
:=
\mathrm{err}_i^-+\mathrm{err}_i^+$ and 
$\alpha^{\mathrm{tot}}
:=
\alpha^-+\alpha^+$. In particular, $\frac{1}{N}\sum_{i=1}^N \mathrm{err}_i^{\mathrm{tot}}
\;\xrightarrow[N\to\infty]{\mathrm{a.s.}}\;
\alpha^{\mathrm{tot}}$.
\end{theorem}
\begin{proof}
We start from the two adaptive recursions defining the one-sided ACI updates:
\[
\alpha_{i+1}^-
=
\alpha_i^- + \gamma\bigl(\alpha^- - \mathrm{err}_i^-\bigr),
\qquad
\alpha_{i+1}^+
=
\alpha_i^+ + \gamma\bigl(\alpha^+ - \mathrm{err}_i^+\bigr).
\]

Summing the two equations yields the recursion for the total adaptive level
$\alpha_i^{\mathrm{tot}} := \alpha_i^- + \alpha_i^+$:
\[
\alpha_{i+1}^{\mathrm{tot}}
=
\alpha_i^{\mathrm{tot}}
+
\gamma\Bigl(
\alpha^- + \alpha^+
-
(\mathrm{err}_i^- + \mathrm{err}_i^+)
\Bigr).
\]

Define
\[
\alpha^{\mathrm{tot}}
:=
\alpha^- + \alpha^+,
\qquad
\mathrm{err}_i^{\mathrm{toi}}
:=
\mathrm{err}_i^- + \mathrm{err}_i^+ .
\]
Then the recursion can be written compactly as
\[
\alpha_{i+1}^{\mathrm{tot}}
=
\alpha_i^{\mathrm{tot}}
+
\gamma\bigl(
\alpha^{\mathrm{tot}}
-
\mathrm{err}_i^{\mathrm{toi}}
\bigr).
\]

Summing both sides from $i=1$ to $N$ gives
\[
\sum_{i=1}^N \alpha_{i+1}^{\mathrm{tot}}
-
\alpha_i^{\mathrm{tot}}
=
\gamma \sum_{i=1}^N
\bigl(
\alpha^{\mathrm{tot}}
-
\mathrm{err}_i^{\mathrm{tot}}
\bigr).
\]
hence
\[
\alpha_{N+1}^{\mathrm{tot}} - \alpha_1^{\mathrm{tot}}
=
\gamma \sum_{i=1}^N
\bigl(
\alpha^{\mathrm{tot}}
-
\mathrm{err}_i^{\mathrm{tot}}
\bigr).
\]
Rearranging terms yields
\[
\frac{1}{N}\sum_{i=1}^N \mathrm{err}_i^{\mathrm{tot}}
-
\alpha^{\mathrm{tot}}
=
\frac{\alpha_1^{\mathrm{tot}} - \alpha_{N+1}^{\mathrm{tot}}}{N\gamma}.
\]

We now bound the numerator on the right-hand side.
By Lemma~\ref{lem:aci_lower}, we have with probability one
\[
\alpha_{N+1}^- \in [-\gamma,\,1+\gamma],
\qquad
\alpha_{N+1}^+ \in [-\gamma,\,1+\gamma].
\]
Summing these bounds gives
\[
\alpha_{N+1}^{\mathrm{tot}}
=
\alpha_{N+1}^- + \alpha_{N+1}^+
\in [-2\gamma,\,2+2\gamma]
\]
with probability one.
Let $\alpha_1^{\mathrm{tot}} = \alpha_1^- + \alpha_1^+$.
Since $\alpha_1^\pm \in [0,1]$, it follows that $\alpha_1^{\mathrm{tot}} \in [0,2]$.
For any $x \in [a,\,b]$ and $S\in\mathbb{R}$, we have
\[
|S - x|
\le
\max\bigl\{ S - a,\ b - S \bigr\}.
\]
Applying this inequality with $x = \alpha_{N+1}^{\mathrm{tot}},  a=-2\gamma, b=2+2\gamma$ and $S=\alpha_1^{\mathrm{tot}}$ yields
\[
\bigl|
\alpha_1^{\mathrm{tot}} - \alpha_{N+1}^{\mathrm{tot}}
\bigr|
\le
\max\Bigl\{
(\alpha_1^- + \alpha_1^+) + 2\gamma,\;
(2+2\gamma) - (\alpha_1^- + \alpha_1^+)
\Bigr\}
\]

Finally, dividing both sides of the previous equation by $N\gamma$ gives
\[
\left|
\frac{1}{N}\sum_{i=1}^N \mathrm{err}_i^{\mathrm{tot}}
-
\alpha^{\mathrm{tot}}
\right|
\le
\frac{
\max\Bigl\{
(\alpha_1^- + \alpha_1^+) + 2\gamma,\;
(2+2\gamma) - (\alpha_1^- + \alpha_1^+)
\Bigr\}
}{N\gamma},
\]
with probability one. This establishes the stated finite-sample bound. Since the right-hand side
vanishes as $N\to\infty$, the almost sure convergence follows immediately.
\end{proof}

\begin{proposition}[Tail-specific coverage guarantees under DtACI]
\label{_prop:dtaci_onesided_longterm}
Let $\gamma^-_{\min} := \min_j\gamma^-_j$ and $\gamma^-_{\max} := \max_j\gamma^-_j$.
Under the lower-tail DtACI procedure with time-varying parameters $\eta^-_i$
and $\sigma^-_i$, with probability one, for every $N \in \mathbb{N}$,
\[
\left|
\frac{1}{N}\sum_{i=1}^N \mathbb E[\mathrm{err}_i^-]-\alpha^-
\right|
\;\le\;
\frac{1+2\gamma^-_{\max}}{N\gamma^-_{\min}}
+
\frac{(1+2\gamma^-_{\max})^2}{\gamma^-_{\min}}
\cdot \frac1N\sum_{i=1}^N \eta^-_i e^{\eta^-_i(1+2\gamma^-_{\max})}
+
2\frac{1+\gamma^-_{\max}}{\gamma^-_{\min}}\cdot \frac1N\sum_{i=1}^N \sigma^-_i.
\]
where the expectation is over the DtACI aggregation probabilities
$\{p_i^{j,-}\}_{j=1}^k$. In particular, if $\eta_i^{-}\to 0$ and $\sigma_i^{-}\to 0$ as $i\to\infty$, then
$\frac{1}{N}\sum_{i=1}^{N}\mathrm{err}_i^{-}
\xrightarrow[N\to\infty]{\mathrm{a.s.}}\alpha^-$.
\end{proposition}
\begin{proof}
The proof follows verbatim the argument of Theorem 6 in \citep{gibbs2024conformal}.
\end{proof}

\begin{theorem}[Global coverage under DtACI]
\label{_thm:dtaci_total_common_gamma}
Let $\mathrm{err}_i^{\mathrm{tot}} := \mathrm{err}_i^- + \mathrm{err}_i^+$
and $\alpha^{\mathrm{tot}} := \alpha^- + \alpha^+$ denote the observed and target miscoverage level of the intersection interval $\mathcal{C}_{i,\underline{\alpha}_i}^{\cap}$, obtained from two independent DtACI procedures run with the same candidate
grid $\{\gamma_j\}_{j=1}^k$ with $\gamma_{\min} := \min_j \gamma_j$ and
$\gamma_{\max} := \max_j \gamma_j$, but possibly different learning-rate
sequences $\{\eta_i^{\pm}\}_{i \ge 1}$ and mixing sequences
$\{\sigma_i^{\pm}\}_{i \ge 1}$. Assume the conformity score functions are quasi-convex
in $y$. Given target levels $\alpha^-, \alpha^+ \in (0,1)$, define
\[
\tilde{\alpha}_i^{\pm} := \sum_{l=1}^k \frac{p_i^{l,\pm}\alpha_i^{l,\pm}}{\gamma_l},
\qquad
\tilde{\alpha}_i^{\mathrm{tot}} := \tilde{\alpha}_i^- + \tilde{\alpha}_i^+,
\qquad
B_0 := \max\!\left\{
  \tilde{\alpha}_1^{\mathrm{tot}} + \frac{2\gamma_{\max}}{\gamma_{\min}},\;
  \frac{2+2\gamma_{\max}}{\gamma_{\min}} - \tilde{\alpha}_1^{\mathrm{tot}}
\right\}.
\]
Then, for every $N\in\mathbb N$,
\begin{align}
\label{_eq:dtaci_total_common_gamma_bound}
\left|
  \frac{1}{N}\sum_{i=1}^{N} \mathbb{E}[\mathrm{err}_i^{\mathrm{tot}}]
  - \alpha^{\mathrm{tot}}
\right|
\;\le\;
\frac{B_0}{N}
&+
\frac{(1+2\gamma_{\max})^2}{\gamma_{\min}}
\cdot\frac{1}{N}\sum_{i=1}^{N}
\Bigl[
  \eta_i^{-}e^{\eta_i^{-}(1+2\gamma_{\max})}
  +
  \eta_i^{+}e^{\eta_i^{+}(1+2\gamma_{\max})}
\Bigr] \nonumber \\
&+
2\,\frac{1+\gamma_{\max}}{\gamma_{\min}}
\cdot\frac{1}{N}\sum_{i=1}^{N}(\sigma_i^{-}+\sigma_i^{+}),
\end{align}
where the expectation is over the DtACI aggregation probabilities
$\{p_i^{j,\pm}\}_{j=1}^k$. In particular, if $\eta_i^{\pm}\to 0$ and $\sigma_i^{\pm}\to 0$ as
$i\to\infty$, then
$\frac{1}{N}\sum_{i=1}^{N}\mathrm{err}_i^{\mathrm{tot}}
\xrightarrow[N\to\infty]{\mathrm{a.s.}}\alpha^{\mathrm{tot}}$.\\ 
When $\eta_i^- = \eta_i^+ =: \eta_i$ and $\sigma_i^- = \sigma_i^+ =: \sigma_i$
for all $i$, Eq.~\eqref{_eq:dtaci_total_common_gamma_bound} simplifies to
\[
\left|
  \frac{1}{N}\sum_{i=1}^{N} \mathbb{E}[\mathrm{err}_i^{\mathrm{tot}}]
  - \alpha^{\mathrm{tot}}
\right|
\;\le\;
\frac{B_0}{N}
+
2\,\frac{(1+2\gamma_{\max})^2}{\gamma_{\min}}
\frac{1}{N}\sum_{i=1}^N
\eta_i\,e^{\eta_i(1+2\gamma_{\max})}
+
4\,\frac{1+\gamma_{\max}}{\gamma_{\min}}
\frac{1}{N}\sum_{i=1}^N \sigma_i.
\]
\end{theorem}
\begin{proof}
To make the proof easy-to-follow, we proceed by steps.
\paragraph*{Step 1}
Let $\alpha_i^{j,\pm}$ denote the level of expert
$j$ at time $i$, and let $p_i^{j,\pm}$ denote the corresponding aggregation
weights.
Let $\tilde{\alpha}^\pm_i := \sum_j \frac{p_i^{j,\pm} \alpha_i^{j,\pm}}{\gamma_j}$ and observe that

\[
\tilde{\alpha}_i^\pm
=
\sum_j
\frac{p_i^{j,\pm}\left(\alpha_{i+1}^{j,\pm} - \gamma_j(\alpha^\pm - \mathrm{err}_i^{j,\pm})\right)}{\gamma_j}
=
\sum_j \frac{p_i^{j,\pm} \alpha_{i+1}^{j,\pm}}{\gamma_j}
+
\sum_j p_i^{j,\pm}(\mathrm{err}_i^{j,\pm} - \alpha^\pm)
\]

\[
=
\tilde{\alpha}_{i+1}
+
\sum_j
\frac{(p_i^{j,\pm} - p_{i+1}^{j,\pm})\alpha_{i+1}^{j,\pm}}{\gamma_j}
+
\sum_j p_i^{j,\pm}(\mathrm{err}_i^{j,\pm} - \alpha^\pm).
\]

Thus,

\begin{equation}\label{eq_exp}
   \mathbb{E}[\mathrm{err}^\pm_i] - \alpha^\pm
=\sum_{j=1}^k p_i^{j,\pm}\mathrm{err}_i^{j,\pm}- \alpha^\pm
=
\tilde{\alpha}^\pm_i - \tilde{\alpha}^\pm_{i+1}
+
\sum_j
\frac{(\tilde{p}_{i+1}^{j,\pm} - p_i^{j,\pm})\alpha_{i+1}^{j,\pm}}{\gamma_j}. 
\end{equation}

For brevity, let's define
\begin{equation*}
\label{eq:R_side_def}
R_i^\pm
:=
\sum_{j=1}^k
\frac{(p_{i+1}^{j,\pm}-p_i^{j,\pm})\alpha_{i+1}^{j,\pm}}{\gamma_j}.
\end{equation*}
\paragraph*{Step 2}
Now, for ease of notation let

\[
W^\pm_i := \sum_j w_i^{j,\pm} \quad \text{and} \quad \tilde{p}_{i+1}^{j,\pm}
=
\frac{p_i^{j,\pm} \exp(-\eta^\pm_i \ell^\pm(\beta^\pm_i,\alpha_i^{j,\pm}))}
{\sum_j p_i^{j,\pm} \exp(-\eta^\pm_i \ell^\pm(\beta^\pm_i,\alpha_i^{j,\pm}))}.
\]

Recall that by definition,

\[
p_{i+1}^{j,\pm}
=
\frac{w_{i+1}^{j,\pm}}{\sum_j w_{i+1}^{j,\pm}}
=
(1-\sigma^\pm_i)\tilde{p}_{i+1}^{j,\pm} + \frac{\sigma^\pm_i}{k}.
\]

Moreover, by a direct computation

\[
\tilde{p}_{i+1}^{j,\pm} - p_i^{j,\pm}
=
\frac{p_i^{j,\pm} \exp(-\eta^\pm_i \ell^\pm(\beta^\pm_i,\alpha_i^{j,\pm}))}
{\sum_j p_i^{j,\pm} \exp(-\eta^\pm_i \ell^\pm(\beta^\pm_i,\alpha_i^{j,\pm}))}
-
p_i^{j,\pm}
\]

\[
=
p_i^{j,\pm}
\frac{
\exp(-\eta^\pm_i \ell^\pm(\beta^\pm_i,\alpha_i^{j,\pm}))
-
\sum_{j'} p_i^{{j',\pm}} \exp(-\eta^\pm_i \ell^\pm(\beta^\pm_i,\alpha_i^{j',\pm}))
}{
\sum_j p_i^{j,\pm} \exp(-\eta^\pm_i \ell^\pm(\beta^\pm_i,\alpha_i^{j,\pm}))
}
\]

\[
=
p_i^{j,\pm}
\frac{
\sum_{j'} p_i^{j',\pm}
\left(
\exp(-\eta^\pm_i \ell^\pm(\beta^\pm_i,\alpha_i^{j,\pm}))
-
\exp(-\eta^\pm_i \ell^\pm(\beta^\pm_i,\alpha_i^{{j',\pm}}))
\right)
}{
\sum_j p_i^{j,\pm} \exp(-\eta^\pm_i \ell^\pm(\beta^\pm_i,\alpha_i^{j,\pm}))
}
\]

\[
=
p_i^{j,\pm}
\frac{
\sum_{j'} p_i^{j',\pm}
\exp(-\eta^\pm_i \ell^\pm(\beta^\pm_i,\alpha_i^{j',\pm}))
\left(
\exp(\eta^\pm_i(\ell^\pm(\beta^\pm_i,\alpha_i^{j',\pm}) - \ell^\pm(\beta^\pm_i,\alpha_i^{i,\pm}))) - 1
\right)
}{
\sum_j p_i^{j,\pm} \exp(-\eta^\pm_i \ell^\pm(\beta^\pm_i,\alpha_i^{j,\pm}))
}
\]

\[
=
p_i^{j,\pm}
\sum_{j'} \tilde{p}_{i+1}^{j',\pm}
\left(
\exp(\eta^\pm_i(\ell^\pm(\beta^\pm_i,\alpha_i^{j',\pm}) - \ell^\pm(\beta^\pm_i,\alpha_i^{j,\pm}))) - 1
\right).
\]

By Lemma~\ref{lem:aci_lower} we know that $\alpha_i^{j,\pm} \in [-\gamma_j,1+\gamma_j]$ and thus that

\[
|\ell^\pm(\beta^\pm_i,\alpha_i^{j',\pm}) - \ell^\pm(\beta^\pm_i,\alpha_i^{j,\pm})|
\le
\max\{\alpha,1-\alpha\} |\alpha_i^{j',\pm} - \alpha_i^{j,\pm}|
\le
1 + 2\gamma_{\max}.
\]

Hence, by the mean value theorem,

\[
\left|
\exp\left(\eta^\pm_i(\ell^\pm(\beta^\pm_i,\alpha_i^{j',\pm}) - \ell^\pm(\beta^\pm_i,\alpha_i^{j,\pm}))\right) - 1
\right|
\le
\eta^\pm_i(1+2\gamma_{\max})
\exp(\eta^\pm_i(1+2\gamma_{\max})).
\]

and thus also

\[
|\tilde{p}_{i+1}^{j,\pm} - p_i^{j,\pm}|
\le
p_i^{j,\pm} \eta^\pm_i(1+2\gamma_{\max})\exp(\eta^\pm_i(1+2\gamma_{\max})).
\]

Applying the triangular inequality to $R_i^\pm$ we get:
\begin{equation}\label{tri_in}
|R_i^\pm|
:=\sum_j
\left|
\frac{(p_{i+1}^{j,\pm} - p_i^{j,\pm})\alpha_{i+1}^{j,\pm}}{\gamma_j}
\right|
\le
(1-\sigma^\pm_i)
\sum_j
\left|
\frac{(\tilde{p}_{i+1}^{j,\pm} - p_i^{j,\pm})\alpha_{i+1}^{j,\pm}}{\gamma_j}
\right|
+
\sigma_i
\sum_j
\left|
\frac{(1/k - p_i^{j,\pm})\alpha_{i+1}^{j,\pm}}{\gamma_j}
\right|
\end{equation}

Applying Lemma~\ref{lem:aci_lower} we have $\alpha_{i+1}^{j,\pm}\in[-\gamma_j,1+\gamma_j]$, which implies that
\[
\left|\frac{\alpha_{i+1}^{j,\pm}}{\gamma_j}\right|
\le
\frac{1+\gamma_j}{\gamma_j}
\le
\frac{1+\gamma_{\max}}{\gamma_{\min}}.
\]
hence, considering that $\sum_j p_i^{j,\pm}=1$ and that $(1-\sigma^\pm_i)<1$ $\forall i$ , we get a bound for the first term of the RHS of eq.\eqref{tri_in}
$$(1-\sigma^\pm_i)
\sum_j
\left|
\frac{(\tilde{p}_{i+1}^{j,\pm} - p_i^{j,\pm})\alpha_{i+1}^{j,\pm}}{\gamma_j}
\right|\le \frac{1+\gamma_{\max}}{\gamma_{\min}}\sum_i\left|
\tilde{p}_{i+1}^{j,\pm} - p_i^{j,\pm}
\right|\le \frac{\eta^\pm_i (1+2\gamma_{\max})^2}{\gamma_{\min}}
\exp(\eta^\pm_i(1+2\gamma_{\max}))$$
Considering now that both \((1/k,\dots,1/k)\) and \((p_i^{1,\pm},\dots,p_i^{k,\pm})\) are probability
vectors, their \(\ell^1\)-distance is at most \(2\):
\[
\sum_{j=1}^k \left|\frac1k-p_i^{j,\pm}\right| \le 2.
\]
hence
\[
\sum_{j=1}^k
\left|
\frac{(\frac1k-p_i^{j,\pm})\alpha_{i+1}^{j,\pm}}{\gamma_j}
\right|
\le
\frac{1+\gamma_{\max}}{\gamma_{\min}}
\sum_{j=1}^k
\left|\frac1k-p_i^{j,\pm}\right|
\le
2\frac{1+\gamma_{\max}}{\gamma_{\min}}.
\]
putting all together we can rewrite get
\begin{equation}\label{eq:dtaci_side_remainder_final}
|R_i^\pm|
:=\sum_j
\left|
\frac{(p_{i+1}^{j,\pm} - p_i^{j,\pm})\alpha_{i+1}^{j,\pm}}{\gamma_j}
\right|
\le
\frac{\eta^\pm_i (1+2\gamma_{\max})^2}{\gamma_{\min}}
\exp(\eta^\pm_i(1+2\gamma_{\max}))
+
2\sigma^\pm_i
\frac{1+\gamma_{\max}}{\gamma_{\min}}.
\end{equation}

\paragraph*{Step 3}
Summing \eqref{eq_exp} over \(s\in\{-,+\}\), and using
\[
\mathrm{err}_i^{tot}=\mathrm{err}_i^-+\mathrm{err}_i^+,
\qquad
\alpha^{tot}=\alpha^-+\alpha^+,
\qquad
\tilde\alpha_i^{tot}=\tilde\alpha_i^-+\tilde\alpha_i^+,
\]
we obtain
$$E[\mathrm{err}_i^{tot}] - \alpha^{tot}
=
\tilde\alpha_i^{tot}-\tilde\alpha_{i+1}^{tot}
+
(R_i^-+R_i^+).$$
Summing over \(i=1,\dots,N\) and noting that $\sum_{i=1}^N\tilde\alpha_i^{tot}-\tilde\alpha_{i+1}^{tot}=\tilde\alpha_1^{tot}-\tilde\alpha_{N+1}^{tot}$ we get
\[
\sum_{i=1}^N \bigl(E[\mathrm{err}_i^{tot}] - \alpha^{tot}\bigr)
=
\tilde\alpha_1^{tot}-\tilde\alpha_{N+1}^{tot}
+
\sum_{i=1}^N(R_i^-+R_i^+).
\]
Dividing by \(N\), taking absolute values and apply triangular inequality gives
\begin{equation}
\label{eq:total_modulus}
\left|
\frac1N\sum_{i=1}^N E[\mathrm{err}_i^{tot}]
-
\alpha^{tot}
\right|
\le
\frac{|\tilde\alpha_1^{tot}-\tilde\alpha_{N+1}^{tot}|}{N}
+
\frac1N\sum_{i=1}^N |R_i^-+R_i^+|.
\end{equation}

\paragraph*{Step 4}
By Lemma~\ref{lem:aci_lower} we have
\[
\tilde\alpha_i^\pm\in
\left[
-\frac{\gamma_{\max}}{\gamma_{\min}},
\frac{1+\gamma_{\max}}{\gamma_{\min}}
\right], \qquad \forall i\ge1.
\]
Which implies that
\[
\tilde\alpha_i^{tot}\in
\left[
-\frac{2\gamma_{\max}}{\gamma_{\min}},
\frac{2+2\gamma_{\max}}{\gamma_{\min}}
\right],
\qquad \forall i\ge1.
\]
In particular,
\[
\tilde\alpha_{N+1}^{tot}\in
\left[
-\frac{2\gamma_{\max}}{\gamma_{\min}},
\frac{2+2\gamma_{\max}}{\gamma_{\min}}
\right].
\]
Now recall the elementary inequality: if \(x\in[a,b]\), then for any \(S\in\mathbb R\),
\[
|S-x|\le \max\{S-a,\; b-S\}.
\]
Applying it with
\[
x=\tilde\alpha_{N+1}^{tot},
\qquad
S=\tilde\alpha_1^{tot},
\qquad
a=-\frac{2\gamma_{\max}}{\gamma_{\min}},
\qquad
b=\frac{2+2\gamma_{\max}}{\gamma_{\min}},
\]
we obtain
\begin{equation}
\label{eq:refined_telescope_bound}
|\tilde\alpha_1^{tot}-\tilde\alpha_{N+1}^{tot}|
\le
\max\left\{
\tilde\alpha_1^{tot}+\frac{2\gamma_{\max}}{\gamma_{\min}},
\;
\frac{2+2\gamma_{\max}}{\gamma_{\min}}-\tilde\alpha_1^{tot}
\right\}
=:B_0.
\end{equation}

\paragraph*{Step 5: bound on the sum of the remainders.}
By the triangle inequality,
\[
|R_i^-+R_i^+|
\le
|R_i^-|+|R_i^+|
\]
using \eqref{eq:dtaci_side_remainder_final} we get
\[
|R_i^-+R_i^+|
\le
\frac{(1+2\gamma_{\max})^2}{\gamma_{\min}}
\Bigl[
\eta_i^{-}e^{\eta_i^{-}(1+2\gamma_{\max})}
+
\eta_i^{+}e^{\eta_i^{+}(1+2\gamma_{\max})}
\Bigr]
\]
\[
\hspace{2.6cm}
+
2\,\frac{1+\gamma_{\max}}{\gamma_{\min}}(\sigma_i^-+\sigma_i^+).
\]
Averaging over \(i=1,\dots,N\), we obtain
\begin{equation}
\label{eq:remainder_average_bound}
\frac1N\sum_{i=1}^N |R_i^-+R_i^+|
\le
\frac{(1+2\gamma_{\max})^2}{\gamma_{\min}}
\frac1N\sum_{i=1}^N
\Bigl[
\eta_i^{-}e^{\eta_i^{-}(1+2\gamma_{\max})}
+
\eta_i^{+}e^{\eta_i^{+}(1+2\gamma_{\max})}
\Bigr]
\end{equation}
\[
\hspace{2.6cm}
+
2\,\frac{1+\gamma_{\max}}{\gamma_{\min}}
\frac1N\sum_{i=1}^N(\sigma_i^-+\sigma_i^+).
\]

\paragraph*{Step 6}
Combining \eqref{eq:total_modulus}, \eqref{eq:refined_telescope_bound}, and
\eqref{eq:remainder_average_bound}, we obtain
\[
\left|
\frac1N\sum_{i=1}^N E[\mathrm{err}_i^{tot}]
-
\alpha^{tot}
\right|
\le
\frac{B_0}{N}
+
\frac{(1+2\gamma_{\max})^2}{\gamma_{\min}}
\frac1N\sum_{i=1}^N
\Bigl[
\eta_i^{-}e^{\eta_i^{-}(1+2\gamma_{\max})}
+
\eta_i^{+}e^{\eta_i^{+}(1+2\gamma_{\max})}
\Bigr]
\]
\[
\hspace{2.6cm}
+
2\,\frac{1+\gamma_{\max}}{\gamma_{\min}}
\frac1N\sum_{i=1}^N(\sigma_i^-+\sigma_i^+),
\]
which is exactly \eqref{_eq:dtaci_total_common_gamma_bound}.
\paragraph*{Step 7: conclusion.}
Finally, if $\eta_i^{\pm}\to 0$ and $\sigma_i^{\pm}\to 0$ as $i \to \infty$, then
\[
\frac{1}{N}\sum_{i=1}^N \eta_i^{\pm}\exp\!\bigl(\eta_i^{\pm}(1+2\gamma_{\max})\bigr)\xrightarrow[N\to\infty]{}0,
\qquad
\frac{1}{N}\sum_{i=1}^N \sigma_i^{\pm}\xrightarrow[N\to\infty]{}0,
\]
which implies that
\[
\lim_{N\to0}\frac{1}{N}\sum_{i=1}^N|\mathrm{err}_i^{\mathrm{tot}}- (\alpha^-+\alpha^+)|=\lim_{N\to0}\frac{1}{N}\sum_{i=1}^N |\mathbb{E}[\mathrm{err}_i^{\mathrm{tot}}]- (\alpha^-+\alpha^+)|=0
\]
where the first equality follows from the law of large numbers.
This completes the proof.
\end{proof}

\begin{proposition}[Tail-specific short-term bounds under DtACI]
\label{_prop:dtaci_onesided_regret}
Let $\{\alpha_i^{*,-}\}_{i \in I}$ be any comparator sequence of oracle miscoverage levels, representing the best adaptive strategy in hindsight over the interval $I$. Let $\gamma^-_{\max} := \max_{1 \le j \le k}\gamma^-_j$
and assume $\gamma^-_1 < \gamma^-_2 < \cdots < \gamma^-_k$ with
$\gamma^-_{j+1}/\gamma^-_j \le 2$ for all $1 < j < k$,
$\gamma^-_k \ge \sqrt{1+1/|I|}$, $\sigma = 1/(2|I|)$,
$\gamma^-_1 \le \sqrt{(\sum_{i=r+1}^s
|\alpha_i^{*,-}-\alpha_{i-1}^{*,-}|+1)/|I|}$, and
$\eta^- = \sqrt{(\log(k|I|)+2)/
\sum_{i=r}^s\mathbb{E}[\ell(\beta_i^-,\alpha_i^-)^2]}$.
Then, 
\begin{enumerate}

\item \textbf{(Dynamic regret)} For any interval $I=[r,s]\subset[N]$,
\begin{align*}
\frac{1}{|I|}\sum_{i=r}^s \mathbb{E}[\ell(\beta^-_i,\alpha^-_i)]
-
\frac{1}{|I|}\sum_{i=r}^s \ell(\beta^-_i,\alpha_i^{*,-})
&\le
2\sqrt{\frac{\log(k|I|)+2}{|I|}}
\sqrt{\frac{1}{|I|}\sum_{i=r}^s\mathbb{E}[\ell(\beta^-_i,\alpha^-_i)^2]}\\
&\quad+
4(1+\gamma^-_{\max})^2
\sqrt{\frac{\sum_{i=r+1}^s|\alpha_i^{*,-}-\alpha_{i-1}^{*,-}|+1}{|I|}}\\
&=
O\!\left(\sqrt{\frac{\log|I|}{|I|}}\right)
+
O\!\left(\sqrt{\frac{\sum_{i=r+1}^s|\alpha_i^{*,-}-\alpha_{i-1}^{*,-}|}{|I|}}\right),
\end{align*}
where the expectation is over the DtACI aggregation probabilities
$\{p_i^{j,-}\}_{j=1}^k$.

\item \textbf{(Coverage deviation)} Additionally assume $\beta^-$ has density $p^-(\cdot)$ on $[0,1]$ with $p^-(x) \ge p_- > 0$, and let $\alpha_i^{\star,-}$ satisfy $\mathbb{P}(Y_i \in \mathcal{C}^L_{i,\alpha_i^{\star,-}}
\mid \{\beta_u^-\}_{u<i}) = 1-\alpha^-$. Then, %combining part~(i) with the bound $|\alpha_i^- - \alpha_i^{\star,-}|^2 \le (2/p^-)\bigl(\mathbb{E}[\ell(\beta^-,\alpha_i^-)] - \mathbb{E}[\ell(\beta^-,\alpha_i^{\star,-})]\bigr)$,
\begin{equation}
\label{_eq:onesided_mse_from_regret}
\frac{1}{|I|}\sum_{i=r}^s
\frac{p_-}{2}\,\mathbb{E}\!\bigl[(\alpha_i^--\alpha_i^{\star,-})^2\bigr]
\;\le\;
O\!\left(\sqrt{\frac{\log|I|}{|I|}}\right)
+
O\!\left(\sqrt{\frac{1}{|I|}\sum_{i=r+1}^s
\mathbb{E}|\alpha_i^{*,-}-\alpha_{i-1}^{*,-}|}\right),
\end{equation}
where the expectation is over $\{p_i^{j,-}\}_{j=1}^k$ and
$\{\beta^-_i\}_{i \le s}$.
\end{enumerate}
\end{proposition}
\begin{proof}
Result~(1) follows from Theorem~4 of \citep{gibbs2024conformal}. We now prove result~(2). By Proposition~5 of \citep{gibbs2024conformal}, we have
\begin{equation}
\mathbb E[\ell(\beta^-,\tau^-)]
-
\mathbb E[\ell(\beta^-,\alpha^{\star,-})]
\;\ge\;
\frac{p_-}{2}(\tau^--\alpha^{\star,-})^2.
\label{one_side_alpha_ineq}
\end{equation}
Combining \eqref{one_side_alpha_ineq} with result~(1) yields:
\begin{equation}
\frac{1}{|I|}
\sum_{i=r}^s
\frac{p_-}{2}
\mathbb E\!\left[
(\alpha_i^- - \alpha_i^{\star,-})^2
\right]
\;\le\;
O\!\left(\sqrt{\frac{\log(|I|)}{|I|}}\right)
+
O\!\left(
\sqrt{
\frac{1}{|I|}
\sum_{i=r+1}^s
\mathbb{E}\!\left[
\left|
\alpha_i^{*,-}-\alpha_{i-1}^{*,-}
\right|
\right]}
\right)
\end{equation}
where the expectation is over $\{p_i^{j,-}\}_{j=1}^k$ and
$\{\beta^-_i\}_{i \le s}$.
\end{proof}

\begin{theorem}[Global short-term bound under DtACI]
\label{_thr:pinball_to_alpha_total_improved}
Let $\alpha_i^{\mathrm{tot}} := \alpha_i^- + \alpha_i^+$ and
$\alpha_i^{*,\mathrm{tot}} := \alpha_i^{*,-} + \alpha_i^{*,+}$ denote the
adaptive and oracle total miscoverage levels for the intersection interval
$\mathcal{C}_{i,\underline{\alpha}_i}^{\cap}$, obtained from two independent
DtACI procedures run in parallel under the same conditions as
Proposition~\ref{_prop:dtaci_onesided_regret}, with conformity score functions
quasi-convex in $y$. Assume $\beta^\pm$ has density $p^\pm(\cdot)$ on
$[0,1]$ satisfying $p^\pm(x) \ge p_\pm > 0$, and that there exists
$\alpha^{\star,\pm}$ with $\mathbb{P}(\beta^\pm < \alpha^{\star,\pm}) =
\alpha^\pm$. Define $p_{\min} := \min\{p_-, p_+\}$ and
$G_i^\pm := \mathbb{E}[\ell(\beta^\pm,\alpha_i^\pm)] -
\mathbb{E}[\ell(\beta^\pm,\alpha_i^{*,\pm})]$.
Then, for every $i$,
\begin{equation}
\label{_eq:prop_total_improved_simpler}
(\alpha_i^{\mathrm{tot}}-\alpha_i^{*,\mathrm{tot}})^2
\;\le\;
\frac{4}{p_{\min}}\bigl(G_i^-+G_i^+\bigr).
\end{equation}
Furthermore, for any interval $I=[r,s]\subset[N]$,
\begin{align}
\label{_eq:short_term_total_like6_improved}
\frac{1}{|I|}\sum_{i=r}^s
\frac{p_{\min}}{4}\,
\mathbb{E}\!\left[(\alpha_i^{\mathrm{tot}}-\alpha_i^{*,\mathrm{tot}})^2
\right]
&\le
O\!\left(\sqrt{\frac{\log|I|}{|I|}}\right) \nonumber \\
&+
O\!\Bigg(
\sqrt{\frac{1}{|I|}\sum_{i=r+1}^s
\mathbb{E}|\alpha_i^{*,-}-\alpha_{i-1}^{*,-}|}
+
\sqrt{\frac{1}{|I|}\sum_{i=r+1}^s
\mathbb{E}|\alpha_i^{*,+}-\alpha_{i-1}^{*,+}|}
\Bigg),
\end{align}
where the expectation is over $\{p_i^{j,\pm}\}_{j=1}^k$ and
$\{\beta_i^\pm\}_{i \le s}$. If additionally the map $\underline{\alpha}
\mapsto \mathbb{P}(Y_i \in \mathcal{C}_{i,\underline{\alpha}}^{\cap} \mid
\{\beta_s^\pm\}_{s<i})$ is $\mathcal{L}_{\cap}$-Lipschitz in
$\alpha^{\mathrm{tot}} = \alpha^- + \alpha^+$, then
\eqref{_eq:short_term_total_like6_improved} also bounds the local coverage
deviation of $\mathcal{C}_{i,\underline{\alpha}_i}^{\cap}$, a condition
reasonable whenever the conformity-score distributions and their quantile
estimates vary smoothly with the miscoverage level.
\end{theorem}
\begin{proof}
Let's start from the identity $\alpha_i^{tot}-\alpha_i^{*,tot}
=
(\alpha_i^- - \alpha_i^{*,-})
+
(\alpha_i^+ - \alpha_i^{*,+})$. Taking absolute values gives $|\alpha_i^{tot}-\alpha_i^{*,tot}|
\le
|\alpha_i^- - \alpha_i^{*,-}|
+
|\alpha_i^+ - \alpha_i^{*,+}|$. From \eqref{one_side_alpha_ineq}, we have for each side
\[
(\alpha_i^\pm-\alpha_i^{*,\pm})^2
\le
\frac{2}{p_\pm}
\left(
\mathbb{E}[\ell(\beta^\pm,\alpha_i^\pm)]
-
\mathbb{E}[\ell(\beta^\pm,\alpha_i^{*,\pm})]
\right)
=
\frac{2G_i^\pm}{p_\pm}.
\]
Hence $|\alpha_i^\pm-\alpha_i^{*,\pm}|
\le
\sqrt{\frac{2G_i^\pm}{p_\pm}}$. Applying the triangle inequality yields $|\alpha_i^{tot}-\alpha_i^{*,tot}|
\le
\sqrt{\frac{2G_i^-}{p_-}}
+
\sqrt{\frac{2G_i^+}{p_+}}$. Squaring both sides, we obtain 
\begin{equation*}
\label{eq:prop_total_improved}
(\alpha_i^{tot}-\alpha_i^{*,tot})^2
\le
\left(
\sqrt{\frac{2G_i^-}{p_-}}
+
\sqrt{\frac{2G_i^+}{p_+}}
\right)^2 .
\end{equation*} 
Finally, the bound \eqref{_eq:prop_total_improved_simpler}
follows by observing that if $a,b \geq0$, then $\left(\sqrt{a}+\sqrt{b}\right)^2
\le
2(a+b)$\footnote{
Indeed,
\(
(\sqrt{a}+\sqrt{b})^2 = a+b+2\sqrt{ab}
\le a+b+(a+b)=2(a+b),
\)
since \(2\sqrt{ab}\le a+b\) (which can be proved by expanding the inequality $\left(\sqrt{a}-\sqrt{b}\right)^2 \geq 0$).
} and using $p_{\min}\le p_-,p_+$.

Combining \eqref{_eq:prop_total_improved_simpler} with result~(1) of Proposition~\ref{_prop:dtaci_onesided_regret}, applied separately to the two tails, yields, for any interval $I=[r,s]\subset[N]$,
\begin{align*}
\frac{1}{|I|}
\sum_{i=r}^s
\frac{p_{\min}}{4}
\mathbb{E}\!\left[
(\alpha_i^{tot}-\alpha_i^{*,tot})^2
\right]
&\le
O\!\left(
\sqrt{\frac{\log(|I|)}{|I|}}
\right)
\nonumber\\
&+
O\!\Bigg(
\sqrt{
\frac{1}{|I|}
\sum_{i=r+1}^s
\mathbb{E}
\bigl|
\alpha_i^{*,-}-\alpha_{i-1}^{*,-}
\bigr|
}+
\sqrt{
\frac{1}{|I|}
\sum_{i=r+1}^s
\mathbb{E}
\bigl|
\alpha_i^{*,+}-\alpha_{i-1}^{*,+}
\bigr|
}
\Bigg).
\end{align*}
where the expectation is over $\{p_i^{j,\pm}\}_{j=1}^k$ and
$\{\beta_i^\pm\}_{i \le s}$.
\end{proof}

\section{Additional Simulation experiments}

\begin{table}[htbp]
\centering
\caption{
Empirical coverage and width of compared methods (non-exchangeable setting),
reported as mean $\pm$ standard deviation over $R = 500$ MC runs. All CP intervals are constructed using the DtACI framework. Global miscoverage level is set at $\alpha = 0.1$, with tail-specific levels equal to $\alpha/2$. Bold font flags global undercoverage, tail undercoverage, or extreme tail overcoverage.
}
\label{tab:cp_faci_ar1}
\scriptsize
\setlength{\tabcolsep}{2.7pt}
\renewcommand{\arraystretch}{1.1}

\begin{tabular}{llcccccc}
\hline
\textbf{Scenario} & \textbf{Method} &
$\widehat{\mathrm{Cov}}$ &
$\widehat{\mathrm{Cov}}_L$ &
$\widehat{\mathrm{Cov}}_U$ &
Mean width &
Median width \\
\hline

%=========================================================
% Normal AR(1)
%=========================================================
Normal AR(1)
& Benchmark
& $\bm{0.898 \pm 0.006}$ & $\bm{0.949 \pm 0.004}$ & $\bm{0.949 \pm 0.004}$
& $3.278 \pm 0.060$ & $3.283 \pm 0.052$ \\
& $\mathcal C_{n+1,\alpha}^{\rm res}$
& $0.902 \pm 0.002$ & $0.951 \pm 0.003$ & $0.951 \pm 0.003$
& $3.336 \pm 0.054$ & $3.304 \pm 0.056$ \\
& $\mathcal C_{n+1,\underline{\alpha}}^{\rm res,\cap}$
& $0.903 \pm 0.002$ & $0.952 \pm 0.001$ & $0.952 \pm 0.001$
& $3.377 \pm 0.059$ & $3.331 \pm 0.057$ \\
& $\mathcal C_{n+1,\alpha}^{\rm s\text{-}res}$
& $0.901 \pm 0.002$ & $0.951 \pm 0.003$ & $0.951 \pm 0.003$
& $3.380 \pm 0.078$ & $3.297 \pm 0.057$ \\
& $\mathcal C_{n+1,\underline{\alpha}}^{\rm s\text{-}res,\cap}$
& $0.903 \pm 0.002$ & $0.952 \pm 0.001$ & $0.952 \pm 0.001$
& $3.463 \pm 0.151$ & $3.324 \pm 0.059$ \\
& $\mathcal C_{n+1,\alpha}^{\rm q}$
& $0.902 \pm 0.002$ & $0.951 \pm 0.003$ & $0.951 \pm 0.003$
& $3.343 \pm 0.055$ & $3.301 \pm 0.057$ \\
& $\mathcal C_{n+1,\underline{\alpha}}^{\rm q,\cap}$
& $0.905 \pm 0.004$ & $0.953 \pm 0.003$ & $0.953 \pm 0.003$
& $3.377 \pm 0.060$ & $3.305 \pm 0.051$ \\
& $\mathcal C_{n+1,\underline{\alpha}}^{\rm q\text{-}sgn,\cap}$
& $0.903 \pm 0.002$ & $0.952 \pm 0.001$ & $0.952 \pm 0.001$
& $3.384 \pm 0.059$ & $3.326 \pm 0.058$ \\[4pt]

%=========================================================
% Student-t AR(1)
%=========================================================
Student-$t$ AR(1)
& Benchmark
& $0.909 \pm 0.007$ & $0.955 \pm 0.005$ & $0.955 \pm 0.004$
& $4.224 \pm 0.148$ & $4.231 \pm 0.133$ \\
& $\mathcal C_{n+1,\alpha}^{\rm res}$
& $0.902 \pm 0.002$ & $0.951 \pm 0.003$ & $0.951 \pm 0.003$
& $4.125 \pm 0.094$ & $4.071 \pm 0.098$ \\
& $\mathcal C_{n+1,\underline{\alpha}}^{\rm res,\cap}$
& $0.903 \pm 0.002$ & $0.952 \pm 0.001$ & $0.952 \pm 0.001$
& $4.213 \pm 0.123$ & $4.121 \pm 0.096$ \\
& $\mathcal C_{n+1,\alpha}^{\rm s\text{-}res}$
& $0.901 \pm 0.002$ & $0.951 \pm 0.003$ & $0.951 \pm 0.003$
& $4.185 \pm 0.113$ & $4.058 \pm 0.097$ \\
& $\mathcal C_{n+1,\underline{\alpha}}^{\rm s\text{-}res,\cap}$
& $0.903 \pm 0.002$ & $0.952 \pm 0.001$ & $0.952 \pm 0.001$
& $4.312 \pm 0.182$ & $4.110 \pm 0.097$ \\
& $\mathcal C_{n+1,\alpha}^{\rm q}$
& $0.902 \pm 0.002$ & $0.951 \pm 0.003$ & $0.951 \pm 0.003$
& $4.139 \pm 0.094$ & $4.066 \pm 0.098$ \\
& $\mathcal C_{n+1,\underline{\alpha}}^{\rm q,\cap}$
& $0.914 \pm 0.006$ & $0.957 \pm 0.004$ & $0.957 \pm 0.004$
& $4.338 \pm 0.158$ & $4.250 \pm 0.129$ \\
& $\mathcal C_{n+1,\underline{\alpha}}^{\rm q\text{-}sgn,\cap}$
& $0.903 \pm 0.002$ & $0.952 \pm 0.001$ & $0.952 \pm 0.001$
& $4.222 \pm 0.121$ & $4.115 \pm 0.097$ \\[4pt]

%=========================================================
% Skew-t AR(1) DtACI
%=========================================================
Skew-$t$ AR(1)
& Benchmark
& $\bm{0.937 \pm 0.006}$ & $\bm{0.937 \pm 0.006}$ & $\bm{1.000 \pm 0.000}$
& $3.020 \pm 0.172$ & $3.023 \pm 0.153$ \\
& $\mathcal C_{n+1,\alpha}^{\rm res}$
& $0.902 \pm 0.002$ & $\bm{0.902 \pm 0.002}$ & $\bm{1.000 \pm 0.001}$
& $2.251 \pm 0.083$ & $2.168 \pm 0.098$ \\
& $\mathcal C_{n+1,\underline{\alpha}}^{\rm res,\cap}$
& $0.903 \pm 0.003$ & $0.952 \pm 0.001$ & $0.952 \pm 0.002$
& $2.662 \pm 0.103$ & $2.580 \pm 0.084$ \\
& $\mathcal C_{n+1,\alpha}^{\rm s\text{-}res}$
& $0.901 \pm 0.002$ & $\bm{0.902 \pm 0.002}$ & $\bm{0.999 \pm 0.001}$
& $2.308 \pm 0.104$ & $2.158 \pm 0.092$ \\
& $\mathcal C_{n+1,\underline{\alpha}}^{\rm s\text{-}res,\cap}$
& $0.903 \pm 0.003$ & $0.952 \pm 0.001$ & $0.952 \pm 0.002$
& $2.755 \pm 0.150$ & $2.573 \pm 0.086$ \\
& $\mathcal C_{n+1,\alpha}^{\rm q}$
& $0.902 \pm 0.002$ & $\bm{0.903 \pm 0.002}$ & $\bm{0.999 \pm 0.001}$
& $2.281 \pm 0.089$ & $2.174 \pm 0.091$ \\
& $\mathcal C_{n+1,\underline{\alpha}}^{\rm q,\cap}$
& $0.949 \pm 0.002$ & $0.950 \pm 0.002$ & $\bm{1.000 \pm 0.000}$
& $3.285 \pm 0.146$ & $3.194 \pm 0.125$ \\
& $\mathcal C_{n+1,\underline{\alpha}}^{\rm q\text{-}sgn,\cap}$
& $0.904 \pm 0.003$ & $0.952 \pm 0.001$ & $0.952 \pm 0.003$
& $2.686 \pm 0.200$ & $2.580 \pm 0.086$ \\

\hline
\end{tabular}
\end{table}

Tables~\ref{tab:cp_faci_ar1} reports the results obtained under the DtACI framework in the non-exchangeable AR(1) setting. Overall, these results confirm the patterns observed for ACI in the main text. Both adaptive aggregation procedures achieve satisfactory global coverage across the Gaussian and Student-$t$ scenarios, while preserving accurate tail-specific calibration for the one-sided intersection intervals. Relative to ACI, DtACI produces very similar intervals, with only a mild increase in width in most scenarios.

The skewed Student-$t$ setting confirms the main conclusions drawn from the simulation experiments based on the ACI procedure. Standard two-sided CP intervals and the benchmark remain miscalibrated at the tail level, displaying lower-tail undercoverage and upper-tail overcoverage. In contrast, the proposed one-sided intersection procedures restore tail-specific calibration by treating the two tails separately. This pattern is stable across ACI and DtACI. As in the results obtained for the ACI procedure, the residual-based intersection method provides a favorable compromise between coverage and efficiency, whereas the quantile-max construction remains highly conservative on the upper tail, with upper-tail coverage equal to one.

\subsection{Technical details on the overcoverage of the $s^{{\rm q}}$ score}

This section complements the results discussion in
 Section~\ref{sec: sim_studies}, giving a technical explanation of the extreme upper-tail overcoverage observed for $\mathcal{C}_{n+1,\underline{\alpha}}^{\rm q,\cap}$ under skewed data. 

The observed extreme overcoverage (equal to $1$) is a direct consequence of the truncation specified by the one-sided quantile scores used to construct $\mathcal C_{n+1,\underline{\alpha}}^{\rm q,\cap}$. Indeed, these
scores are defined as
\[
s^{{\rm q},{\rm L}}(x,y)
=
\max\!\left\{
\hat q_{\alpha,n_{t}}(x)-y,0
\right\},
\qquad
s^{{\rm q},{\rm U}}(x,y)
=
\max\!\left\{
y-\hat q_{1-\alpha,n_{t}}(x),0
\right\}.
\]
Hence, the corresponding calibration quantiles
$Q_{\alpha,n_{c}}^{{\rm q},{\rm L}}$ and
$Q_{\alpha,n_{c}}^{{\rm q},{\rm U}}$ are always non-negative. As shown in Section~\ref{sec:onesided_split_cp}, the induced one-sided bounds are
\begin{equation}\label{LUMAX}
L^{\rm q}(x)
=
\hat q_{\alpha,n_{t}}(x)
-
Q_{\alpha,n_{c}}^{{\rm q},{\rm L}},
\qquad
U^{\rm q}(x)
=
\hat q_{1-\alpha,n_{t}}(x)
+
Q_{\alpha,n_{c}}^{{\rm q},{\rm U}}.
\end{equation}

In negatively skewed distributions, observations are much more frequently located
below the upper conditional quantile $\hat q_{1-\alpha,n_{t}}(x)$. Consequently, for the upper-tail score, many calibration observations satisfy
$y \le \hat q_{1-\alpha,n_{t}}(x)$ and therefore produce many
$s^{{\rm q},{\rm U}}(x,y)=0$. 

The presence of many zeroes upper-tail scores generates ties. Hence, the almost-sure distinctness assumption required for the finite-sample upper coverage bounds in Theorem~\ref{_thm:two_sided_coverage} is no longer satisfied. This makes extreme upper-tail overcoverage possible and is consistent with the empirical behavior of $\mathcal C_{n+1,\underline{\alpha}}^{\rm q,\cap}$, whose upper-tail coverage is equal to one. We now explain why this overcoverage is so severe.

The concentration of upper-tail scores at zero lowers the empirical calibration
quantile $Q_{\alpha,n_{c}}^{{\rm q},{\rm U}}$. However, because of the
truncation induced by the maximum operator, this quantile can decrease at most
to zero and can never become negative. As a result, the upper endpoint in \eqref{LUMAX} is bounded from below by the upper conditional quantile
$\hat q_{1-\alpha,n_{t}}(x)$. This restriction prevents the upper bound $U^{\rm q}(x)$ from moving sufficiently downward in settings where observations rarely exceed the upper conditional quantile. The truncated construction therefore yields an upper bound that
is systematically too conservative under strong negative skewness, which explains
the empirical upper-tail coverage equal to one. In contrast, the non-truncated
score $s^{{\rm q\text{-}sgn},{\rm U}}(x,y)
=
y-\hat q_{1-\alpha,n_{t}}(x)$ can take negative values when
$y<\hat q_{1-\alpha,n_{t}}(x)$, allowing its calibration quantile to be
negative and, consequently, allowing the upper endpoint of $\mathcal C_{n+1,\underline{\alpha}}^{\rm q\text{-}sgn,\cap}$ to move below $\hat q_{1-\alpha,n_{t}}(x)$. 

\section{Additional Results on Upper-Tail Coverage Dynamics}
Figure~\ref{fig:upper_coverage} reports the average empirical upper-tail coverage over time for SPY, TQQQ, and XLE. Consistently with the lower-tail analysis presented in the main text, both the benchmark GARCH-$t$ model and the classical two-sided conformal procedures exhibit systematic directional miscalibration, producing persistent overcoverage relative to the nominal upper-tail target level $\alpha^+$. By contrast, the proposed one-sided conformal procedures achieve substantially improved alignment with the desired upper-tail coverage level across all considered assets. Overall, these results further support the importance of independently calibrating the two tails, especially in asymmetric and nonstationary financial environments where lower- and upper-tail risks may evolve differently over time.

%=========================================================
% MULTIPANEL — UPPER COVERAGE
%=========================================================
\begin{figure}[htbp]
\centering

\includegraphics[
width=0.72\textwidth
]{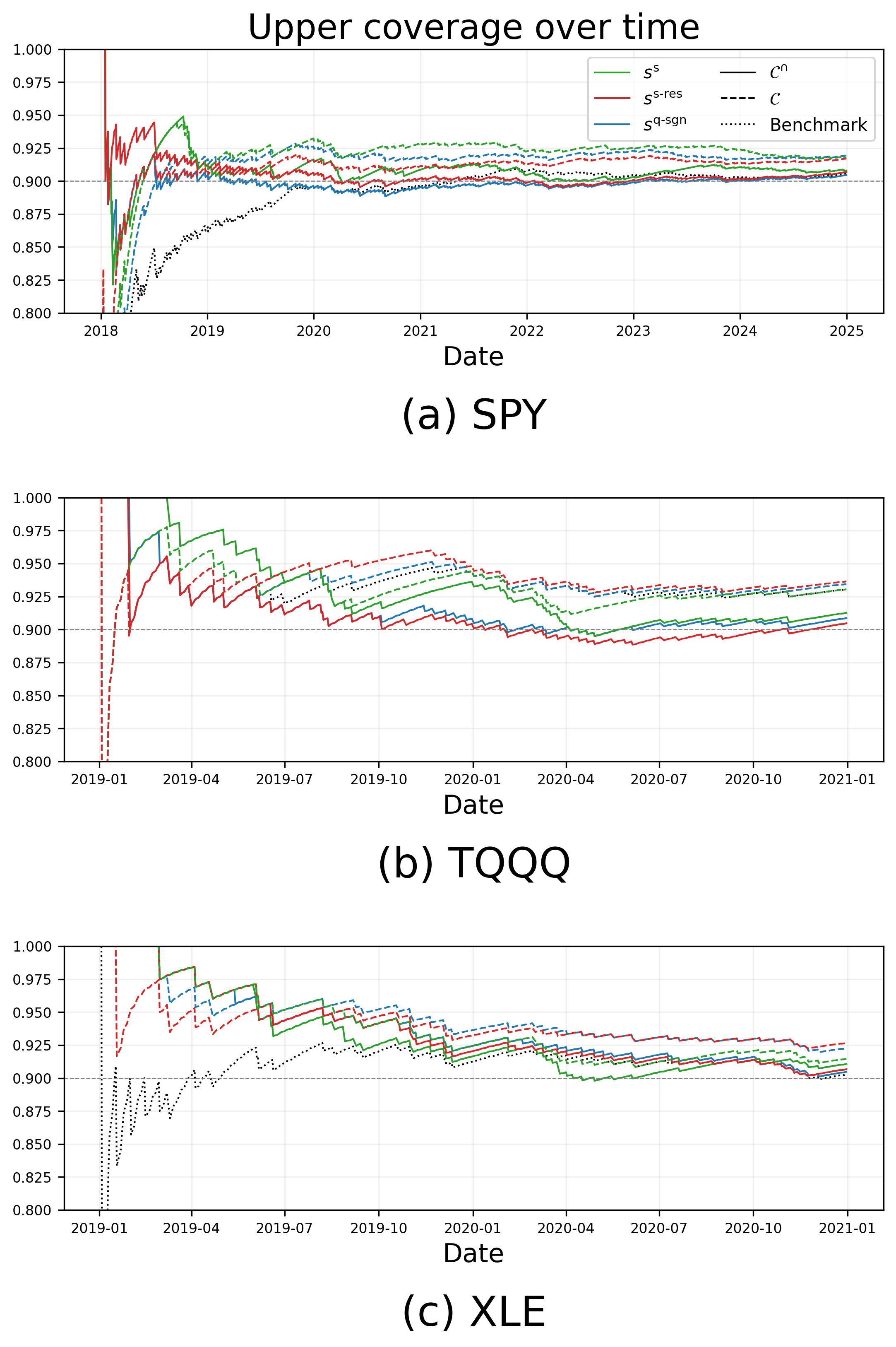}

\caption{
Comparison of upper-tail coverage over time for SPY, TQQQ, and XLE obtained using the benchmark forecast model GARCH-$t$ in isolation and in combination with the CP procedures under the residual $s^{\text{res}}$ (res), s-residual $s^{\text{s-res}}$ (s-res), and quantile scores $s^{\text{q-sgn}}$ (q-sgn).}
\label{fig:upper_coverage}

\end{figure}

\end{document}